\documentclass[11pt,twoside]{article}
\usepackage{amssymb,amsfonts,amsmath,amsthm}%
\usepackage{esint}%
\usepackage{graphicx}%
\usepackage{color}%
\newcommand{\figdraft}{false}%
\newcommand{\figfile}[1]{#1}%
\newcommand{\figwidth}{0.425\textwidth}%
\theoremstyle{plain}%
\newtheorem{theorem}{Theorem}[]%
\newtheorem{corollary}[theorem]{Corollary}%
\newtheorem{lemma}[theorem]{Lemma}%
\newtheorem{remark}[theorem]{Remark}%
\def\longrightharpoonup{
\relbar\joinrel\joinrel\relbar\joinrel\joinrel\relbar\joinrel\joinrel\rightharpoonup}
\newcommand{\xrightharpoonup}[1]{\stackrel{#1}{\longrightharpoonup}}

\newcommand{\sgn}{\mathrm{sgn}}

\newcommand{\fspace}[1]{{\mathsf{#1}}}
\newcommand{\fspaceL}{\fspace{L}}
\newcommand{\fspaceC}{\fspace{C}}
\newcommand{\fspaceW}{\fspace{W}}

\newcommand{\ol}[1]{{\overline{#1}}}

\newcommand{\Rset}{{\mathbb{R}}}

\newcommand{\Nset}{{\mathbb{N}}}

\newcommand{\ocinterval}[2]{(#1,\,#2]}%
\newcommand{\cointerval}[2]{[#1,\,#2)}%
\newcommand{\oointerval}[2]{(#1,\,#2)}%
\newcommand{\ccinterval}[2]{[#1,\,#2]}%

\newcommand{\DO}[1]{{O\at{#1}}}
\newcommand{\Do}[1]{{o\at{#1}}}

\newcommand{\nom}{{\rm num}}
\newcommand{\denom}{{\rm den}}

\newcommand{\tdots}{{...}}%

\newlength{\mhpicDwidth}
\newlength{\mhpicDvsep}
\newlength{\mhpicDhsep}

\newlength{\mhpicPwidth}
\newlength{\mhpicPvsep}
\newlength{\mhpicPhsep}

\newlength{\mhpicWhsep}

\newcommand{\pair}[2]{{\left({#1},\,{#2}\right)}}

\newcommand{\at}[1]{{\left({#1}\right)}}
\newcommand{\nat}[1]{(#1)}
\newcommand{\bat}[1]{{\big(#1\big)}}
\newcommand{\Bat}[1]{{\Big(#1\Big)}}

\newcommand{\bigpar}{\par\quad\newline\noindent}

\newcommand{\jump}[1]{{|\![#1]\!|}}

\newcommand{\abs}[1]{\left|{#1}\right|}
\newcommand{\babs}[1]{\big|{#1}\big|}
\newcommand{\Babs}[1]{\Big|{#1}\Big|}

\newcommand{\dint}[1]{\,\mathrm{d}#1}

\newcommand{\Ga}{{\Gamma}}

\newcommand{\eps}{{\varepsilon}}

\newcommand{\la}{{\lambda}}

\newcommand{\calU}{\mathcal{U}}

\usepackage[%
a4paper,%
nohead,%
twoside,
ignorefoot,
scale=0.8,
bindingoffset=1.5cm
]{geometry}%
\usepackage{float}%
\newcommand{\ignore}[1]{}

\begin{document}%
%
%
\title{Self-similar Solutions to a Kinetic Model for Grain Growth}
\renewcommand{\thefootnote}{}%
\footnotetext{This work was supported by the Royal Society and the CNRS through the International Joint Project
JP 090230 and through the EPSRC Science and Innovation award to the Oxford Centre for
Nonlinear PDE (EP/E035027/1).}%
\date{\today}%
\author{%
Michael Herrmann\thanks{Universit\"at des Saarlandes, Fachrichtung Mathematik, 
Postfach 151150, D-66041 Saarbr\"ucken, Germany,
{\tt michael.herrmann@math.uni-sb.de}}
 \and Philippe Lauren\c{c}ot\thanks{Institut de Math\'ematiques de Toulouse, CNRS UMR~5219, 
Universit\'e de Toulouse, F--31062 Toulouse Cedex 9, France,
{\tt laurenco@math.univ-toulouse.fr}}
 \and Barbara Niethammer\thanks{%
        Oxford Centre of Nonlinear PDE, University of Oxford, 24-29 St Giles', OX1 3LB, United Kingdom, %
        {\tt niethammer@maths.ox.ac.uk}}
}%
\maketitle
%
%
%
\begin{abstract}%
We prove the existence of self-similar solutions to the Fradkov model for two-dimensional grain growth, which
consists of an infinite number of  nonlocally coupled transport
equations for the number densities of grains with given
area and number of neighbours (topological class). 
For the proof we introduce a finite maximal topological class and
study an appropriate upwind-discretization of the time dependent
 problem in self-similar variables. We 
 first show that the resulting finite dimensional differential system has nontrivial steady states. 
Afterwards we let the
discretization parameter tend to zero and prove that the steady states converge to a  compactly supported
self-similar solution for a Fradkov model with finitely many equations. In a third step we let 
the maximal topology class tend to infinity and obtain 
self-similar solutions to the original system that decay exponentially. 
Finally, we use the upwind discretization to compute self-similar solutions numerically.

\end{abstract}%
%
%
\quad\newline\noindent%
\begin{minipage}[t]{0.15\textwidth}%
Keywords: %
\end{minipage}%
\begin{minipage}[t]{0.8\textwidth}%
\emph{grain growth, kinetic model, self-similar solution} %
\end{minipage}%
\medskip
\newline\noindent
\begin{minipage}[t]{0.15\textwidth}%
MSC (2010):  %
\end{minipage}%
\begin{minipage}[t]{0.8\textwidth}%
34A12, 35F25, 35Q82, 74A50
\end{minipage}%
%
%
%
%
%
\setcounter{tocdepth}{5} %
\setcounter{secnumdepth}{4}
{\scriptsize{\tableofcontents}}%
%
\section{Introduction}
\label{S.intro}

Grain growth denotes the late stage coarsening of polycrystalline materials when
migration of grain boundaries due to
capillary forces causes small grains to vanish and larger grains to grow.
One often observes that despite having different histories many materials
eventually exhibit  universal statistically self-similar coarsening behaviour,
usually referred to as normal grain growth. 
Different approaches have been used to predict and explain this
phenomenon, see \cite{ASGS84a,ASGS84b} for Monte-Carlo methods,
\cite{KNN89} for a study of vertex models, and more recently \cite{Elseyetal2009,Elseyetal2011} for boundary tracking methods. We also refer to the  theoretical approach in \cite{barmaketal2011a,barmaketal2011b}, which is based on the grain boundary character distribution.
\par
However, it remains a challenge  to establish such universal long-time asymptotics in mathematical
models and it is often  difficult to prove only the existence of self-similar solutions.
 In this article we investigate the existence of self-similar
solutions to a kinetic mean-field model that has
been suggested by Fradkov \cite{Fr88a} to describe grain growth in two dimensions.

\subsection{Fradkov's mean-field model}

We briefly describe the derivation of Fradkov's model and refer 
 to \cite{Fr88a,FU94,HHNV08} for more details. Our starting point are two-dimensional periodic networks of grain boundaries that  meet in triple junctions
(Fig. \ref{fig:isocrystals}).  In the case of constant surface energy and infinite mobility of triple junctions, 
the grain boundaries move according to the
mean curvature flow while all angles at the triple junctions
are 
$2\pi/3$. 
\begin{figure}[h!]
\centering%
{\includegraphics[width=0.5\textwidth]{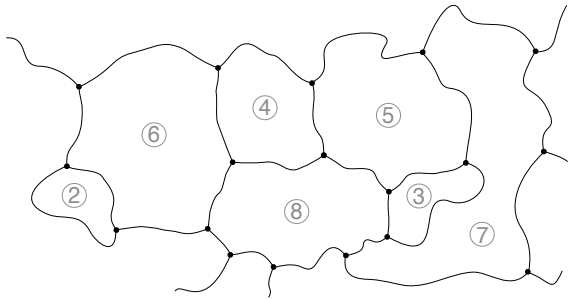}}%
\caption{Cartoon of a 2D network of grains with triple junctions, where the Herring condition implies that all angles are
equal to $2\pi/3$. The encircled numbers refer to the topology classes.} \label{fig:isocrystals}
\end{figure}
In this setting one can easily derive the von Neumann--Mullins law
for the area
$a(t)$ at time $t>0$ of a single grain with $n$ edges \cite{M56}:
\begin{align}
\label{vNMl} \frac{\dint{}}{\dint{t}} a\left(t\right) = M \sigma
\frac{\pi}{3} \left(n-6\right)\,.
\end{align}
Here $M$ denotes the mobility of the grain boundaries and $\sigma$
the surface tension. 
\par%
\begin{figure}[ht]
\centering%
{\includegraphics[width=0.4\textwidth]{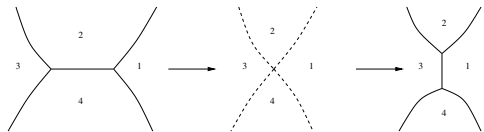}}%
\caption{Neighbour switching} \label{fig:switching}
\end{figure}
The evolution of such a network by mean curvature is well--defined 
\cite{KL01,MNT04} until two vertices
on a grain boundary collide, after which topological rearrangements
may take place. If an edge vanishes an unstable fourfold vertex is produced,
which immediately splits up again such that two
new vertices are connected by a new edge. As a consequence two neighbouring grains decrease their topological
class (i.e., the number of edges), whereas the other two grains
increase it (Fig. \ref{fig:switching}). Furthermore,  grains
can vanish such that some vertices and edges disappear.
Due to the von Neumann--Mullins law this can only happen for
grains  with
topological class $2 \le n \le 5$.
As illustrated in Fig. \ref{fig:vanishing},
the vanishing of a grain of topological class $n=4$ or $n=5$ can result
in topologically different configurations.
\begin{figure}[ht]
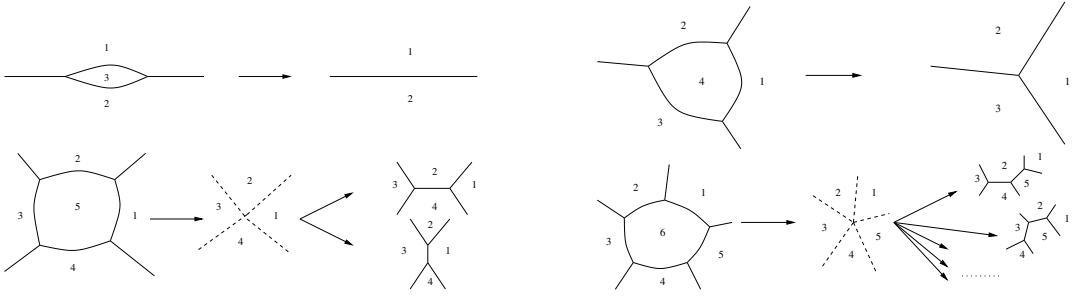

\centering%
{%
\begin{minipage}[c]{0.4\textwidth}%
\includegraphics[width=\textwidth]{\figfile{grainvanishinga}}%
\end{minipage}%
\hspace{0.1\textwidth}%
\begin{minipage}[c]{0.4\textwidth}%
\includegraphics[width=\textwidth]{\figfile{grainvanishingb}}%
\end{minipage}%
\\%
\begin{minipage}[c]{0.4\textwidth}%
\includegraphics[width=\textwidth]{\figfile{grainvanishingc}}%
\end{minipage}%
\hspace{0.1\textwidth}%
\begin{minipage}[c]{0.4\textwidth}%
\includegraphics[width=\textwidth]{\figfile{grainvanishingd}}%
\end{minipage}%
}%
\caption{Grain vanishing} \label{fig:vanishing}
\end{figure}
\par%
In order to derive a kinetic description 
of the evolution of the grain boundary network we introduce
the number densities ${f_{n}\left(a,t\right)}$ of 
grains with topological class $n\geq2$ and area $a\geq0$ at time
$t\geq0$. As long as no topological rearrangements take place,
the   von Neumann--Mullins law \eqref{vNMl} implies that $f_n$ evolves
according to
\begin{align*}
\partial_{t} {f_{n}\left(a,t\right)} + \left(n-6\right) \partial_{a} {f_{n}\left(a,t\right)} = 0\,.
\end{align*}
This equation needs to be supplemented with boundary conditions at $a=0$ for $n>6$. 
It is reasonable to assume  that no new grains are created during the coarsening process, which
implies that
\begin{equation}\label{boundarycond}
{f_{n}\left(0,t\right)} = 0 \quad\quad\text{for}\quad  n \geq7\,.
\end{equation}
To model the topological changes we define a collision operator $\tilde{J}$ that couples
the equations for different topological classes. More precisely, we introduce
topological fluxes $\eta_{n}^{+}$ and $\eta_{n}^{-}$ that describe the
flux from class $n$ to $n+1$ and from $n$ to $n-1$, respectively, and set
\begin{equation*}
\nat{\tilde{J}f}_{n} = \eta_{n-1}^{+} + \eta_{n+1}^{-} -
\eta_{n}^{+} - \eta_{n}^{-}\,
\end{equation*}
with $\eta_1^+=\eta_2^-=0$ due to $n\geq2$.
Employing  a mean-field assumption Fradkov \cite{Fr88a} suggests that the fluxes are given by
\begin{equation}
\label{intro.Fluxes}
\eta_{n}^{+} = \Gamma\beta\, n f_{n}\,,\qquad \eta_{n}^{-} =
\Gamma\left(\beta+1\right) n f_{n},
\end{equation}
where the coupling weight $\Gamma$ describes the intensity of topological changes and depends on the complete state of the system in a
self--consistent way, see \eqref{Gammadef} below. Moreover, the parameter
 $\beta$ measures the ratio between switching events and vanishing events. Our analysis requires $\beta \in (0,2)$, but the numerical simulations work well also for larger $\beta$. 
\par%
Although Fradkov's model has no upper bound for the topology class $n$, it is convenient for the mathematical
analysis to consider variants of the model with $2\leq{n}\leq{N}$ for some $6<N<\infty$. In this case we close the equations for 
$(f_n)_{2\leq n\leq N}$ by assuming $\eta_{N+1}^-=\eta_N^+=0$.
\par%
Assumption \eqref{intro.Fluxes} implies that the collision terms are given by
$\tilde{J}f=\Gamma{Jf}$ with 
\begin{align}
\label{intro.DefJ}
\begin{split}
\at{Jf}_{2} &= 3\left(\beta+1\right)f_{3}-2\beta f_{2},
\\%
\at{Jf}_{n} &= \left(\beta+1\right)\left(n+1\right)f_{n+1}-
\left(2\beta+1\right)n\,f_{n}+\beta\left(n-1\right)f_{n-1}
\quad\text{for}\quad{}2<n<N-1,
\\%
\at{Jf}_{N} &= \beta (N{-}1) f_{N{-}1} - (\beta +1) N f_N,
\end{split}
\end{align}
where the last identify is not used for $N=\infty$. Notice that this definition ensures the zero balance property
\begin{align}
\notag
\sum\limits_{n=2}^N\left(Jf\right)_{n}\pair{a}{t} = 0\qquad\text{for all}\quad a,\,t
> 0,
\end{align}
which reflects that the number of grains with given area does not change
due to  switching or vanishing events. To summarize, 
 the kinetic model we consider in this paper is  given by
\begin{align}
\label{infinitesystem}%
\partial_{t} {f_{n}\left(a,t\right)} +
 \left(n-6\right) \partial_{a} {f_{n}\left(a,t\right)} &=
\Gamma\bat{f\at{t}}{\left(J
f\right)_{n}\left(a,t\right)}\,,\qquad (a,t) \in (0,\infty)^2\,, \; n\geq 2, 
\end{align}
with boundary conditions
\eqref{boundarycond}, either $N<\infty$ or $N=\infty$, and $(Jf)_n$ given by (\ref{intro.DefJ}). 
\bigpar%
It remains to determine the coupling weight $\Gamma$ in dependence of $f$. The key idea is to choose $\Gamma$ such that the total  area 
\begin{align*}
A(t)= \sum_{n=2}^NY_n(t) \qquad \mbox{ with } \qquad Y_n(t)=\int_0^{\infty} a f_n(a,t)\,\dint{a}
\end{align*}
 is conserved during the evolution.  One easily checks that $dA/dt = P$, where $P$ is the polyhedral defect defined by  
\begin{align*}
P(t)=\sum_{n=2}^{N} \at{n-6} X_n(t) \qquad \mbox{ with } \qquad X_n(t)=\int_0^{\infty} f_n(a,t)\,\dint{a}.
\end{align*}
 The polyhedral formula 
$P=0$ resembles Euler's formula for networks with triple junctions and states that the average number of neighbours per grain is $6$.
 We now readily verify that $dP/dt = 0$ holds if and only if
\begin{equation}
\label{Gammadef}
\Gamma \big( f(t)\big)=\frac{ 
\sum\limits_{n=2}^5\left(n-6\right)^2
f_{n}(0,t)}{
\beta{N}X_N(t)- 2\left(\beta+1\right)X_2(t)+\sum\limits_{n=2}^Nn X_n(t)}
\,,
\end{equation}
where we use the convention $NX_N=0$ for $N=\infty$. In particular, \eqref{Gammadef} guarantees
the conservation of area and the polyhedral formula provided that the initial data satisfy $P=0$.

We finally mention that well-posedness of Fradkov's model, both for $N<\infty$ and $N=\infty$, has been established in \cite{HHNV08} for 
$\beta \in (0,2)$. A similar model with finite $N$ 
has been  considered in \cite{Co09}.
%
\subsection{Self-similar solutions and main result}

Self-similar solutions to \eqref{infinitesystem} take the form
\begin{equation}
\notag\label{sseq}
f_n(a,t) = \frac{g_n\at\xi}{t^2} \,, \qquad \xi = \frac{a}{t}\geq0\,,
\end{equation}
where the sequence $g=(g_n)_{n \geq 2}$ of self-similar profiles satisfies 
\begin{equation}
\label{ss1}
- 2 g_n - \big( \xi + 6-n\big)  g_n^\prime = \Gamma \big( Jg\big)_n\,
\end{equation}
for some positive constant $\Gamma$ 
as well as the boundary conditions $g_n\at{0}=0$ for $n> 6$.
With some abuse of notation, we define the moments
\[X_n=\int_0^{\infty} g_n(\xi)\,d\xi \qquad \mbox{ and } \qquad Y_n=\int_0^{\infty} \xi g_n(\xi)\,d\xi,
 \]
and refer to  
\[ P=\sum_{n= 2}^N \at{n-6}X_n \qquad \mbox{ and } \qquad A = \sum_{n= 2}^N Y_n
 \]
 as 
 the polyhedral defect and the area of a self-similar solution, respectively.
\par%
It is important to note that each sufficiently integrable solution to \eqref{sseq} satisfies
 the analogue of the polyhedral formula, and that the coupling weight $\Gamma$ depends on the $g_n$'s in a self-consistent manner.
In fact,  multiplying
\eqref{ss1} by $\xi$, integrating with  respect to $\xi$, and summing over $n$, we find that the zero balance property of $J$ implies $P=0$. Similarly, 
if we multiply by $1$ instead of $\xi$, we easily
derive the analogue to \eqref{Gammadef}, that means we have
$\Gamma\at{g}=\Gamma_\nom\at{g}/\Gamma_\denom\at{g}$ with
\begin{equation}\label{Gammadef2}
 \Gamma_\nom\at{g}= 
\sum\limits_{n=2}^5\left(n-6\right)^2
g_{n}(0)\,,\qquad
\Gamma\at{g}= 
\beta{N}X_N- 2\left(\beta+1\right)X_2+\sum\limits_{n=2}^Nn X_n\,.
\end{equation}
The main mathematical difficulty in the existence proof for self-similar solutions
stems from the fact that the ordinary differential equation
\eqref{ss1} is singular at $\xi=n-6$ and has different transport directions for $\xi<n-6$ and $\xi>n-6$. 
In this paper we prove the existence of \emph{weak} self-similar solutions for both $N<\infty$ and $N=\infty$, where weak solution means that each function $g_n$ satisfies
\begin{equation}
\label{weakformulation}
\int_0^{\infty} g_n\bat{\at{\xi + 6-n}\phi' - \phi}  \dint{\xi}
+ (6-n)_+\, g_n(0) \phi(0)=
\Gamma(g) \int_0^{\infty} \big( Jg\big)_n \phi \dint{\xi}
\end{equation}
for all smooth test functions $\phi$ with compact support in $\cointerval{0}{\infty}$.
\bigpar
Our existence result can be summarized as follows.
\begin{theorem}\label{Tmain1}
Let $\beta \in (0,2)$ and assume that either  $6<N<\infty$ or $N=\infty$. Then, there exists a weak self-similar solution to the Fradkov model that is nontrivial and 
nonnegative with finite area, and satisfies
\begin{align*}
 \sum_{n=2}^N \Big ( e^{\lambda n}  X_n + \int_0^{\infty} e^{\lambda \xi} g_n(\xi)\,d\xi \Big )<\infty 
\end{align*}
for all $0<\la<\ln\at{1+1/\beta}$. Moreover,  for $N<\infty$ 
all functions $g_n$ are supported in $\ccinterval{0}{N-6}$. 
\end{theorem}
These main assertions can be supplemented by the following remarks.
\begin{enumerate}
\item  
Since $\Gamma(g)$ depends on $g$ homogeneously of order $0$, the set of self-similar solutions is invariant under scalings $g_n\rightsquigarrow\la{g_n}$ with $\la>0$. Therefore we can normalize self-similar solutions by prescribing the area.
\item 
The weak formulation combined with the integrability condition $g_n\in\fspaceL^1(0,\infty)$ implies  regularity results. Specifically, in a first step we find that each function $g_n$ is continuous at all points $\xi\in\cointerval{0}{\infty}\setminus\{n-6\}$. Using this we then easily show in a second step that $g_n$ is even continuously differentiable
at all points $\xi\in\cointerval{0}{\infty}\setminus\{n-7,n-6,n-5\}$. 
We also  establish further regularity results that characterize the behaviour of $g_n$ near $\xi=n-6$, see Lemma \ref{Lem:NearSingularities}, and find for large $n$ that $g_n$ is continuous also at $\xi=n-6$.
\item 
The moment estimates from Theorem \ref{Tmain1} imply that $g_n$ decays exponentially in $\xi$. More precisely, multiplying
\eqref{ss1} with $e^{\la\xi}$ and integrating over $\cointerval{\xi}{\infty}$ gives
$g_n\at\xi\leq{C}_ne^{\la\xi}$ for some constant $C_n$ and all $\xi$.
\item  
Numerical simulations as described in Section \ref{S.numerics} indicate, at least for $N<\infty$, 
that for each $\beta$ there exists a unique self-similar solution with prescribed area, but we are not able to prove this.
\end{enumerate}
Our strategy for proving Theorem \ref{Tmain1} is inspired 
by the existence proof for self-similar solutions to coagulation equations  in \cite{FL05}.
In  Section \ref{S.DiscreteModel} we introduce a finite-dimensional dynamical model that can be regarded as a semi-discrete upwind scheme
for \eqref{infinitesystem} in self-similar variables, and involves the discretization length $0<\eps\ll1$. To derive this scheme we assume that $N<\infty$, restrict
the rescaled area variable $\xi$ to a finite domain $\ccinterval{0}{L}$ with sufficiently large $L$, and impose artificial Dirichlet conditions at $\xi=L$. Moreover, we identify
a discrete analogue to \eqref{Gammadef} that ensures the conservation of both area and polyhedral formula. Standard results 
from the theory of dynamical systems then imply the existence of nontrivial 
steady states for each sufficiently small $\eps$. 
\par%
Afterwards we show that these steady states converge as $\eps\to0$ to a self-similar profile for the Fradkov model for $N<\infty$. The proof of this assertion combines two different arguments: First, in 
Section \ref{S.Limit1} we derive  suitable a priori estimates that allow us to extract convergent subsequences 
whose limits provide candidates for the self-similar profiles. Second, in Section \ref{S.Limit2} we analyse the behaviour near the singular points $\xi=n-6$ in order to rule out that Dirac masses appear in the limit.
\par
In Section \ref{S.Limit3} we establish the exponential decay of $X_n$ and derive uniform estimates for higher moments. The resulting tightness estimates then enable us to pass to the 
limit $N\to\infty$ in Section \ref{S.Limit4}. 
Finally, in Section \ref{S.numerics} we illustrate that
the upwind discretization of \eqref{infinitesystem} in self-similar variables, combined with explicit Euler steps for the time discretization, provides a convenient algorithm for the numerical computations of self-similar solutions. 
%
\section{The discrete dynamical model}\label{S.DiscreteModel}
%
In order to prove the existence of self-similar solutions we study an upwind
finite-difference discretization of the time-dependent problem in self-similar variables. 
To that aim  we restrict the rescaled area variable $\xi$ to a finite interval $[0,L]$ with $L\in\Nset$ and $L>{N-6}$, and for each $K\in\Nset L$ we consider the grid points $\xi_k=kL/K$, such that
\begin{align*}
\xi_{k+1}-\xi_k=L/K=:\eps.
\end{align*}
Notice that for each $n>6$ the critical area $\xi={n-6}$ corresponds to one of the grid points, that means for each $n\geq2$ we have 
\begin{align*}
\xi_{k_n}=n-6,\qquad k_n:=\at{n-6}K/L
\end{align*}
with $k_6=0$, $k_n<0$ for $n=2\tdots{5}$, and $0<k_n<K$ for $n=7\tdots{N}$.
\par
Using the difference operators
$\nabla^-$ and $\nabla^+$ with 
\begin{align*}
\nabla^+u^k=\frac{u^{k+1}-u^k}{\eps},\qquad\nabla^-u^k=\frac{u^{k}-u^{k-1}}{\eps},
\qquad
\end{align*}
we mimic the transport term $-\at{\xi+6-n}g_n^\prime\at\xi$ by the upwind discretization 
\begin{align}
\label{Eqn:Transportoperator}
-\at{\xi_k+6-n}_+\nabla^+g_n^k+
\at{\xi_k+6-n}_-\nabla^-g_n^k+\delta_k^{k_n}{g_n^k}.
\end{align}
Here $\delta_k^{k_n}$ is 
the usual Kronecker delta and ${x}_\pm$ denotes the positive and negative part of $x$, that means
${x}_\pm=\max\{\pm{x},\,0\}\geq0$ and $x={x}_+-{x}_-$.
\par
At a first glance, the Kronecker delta in \eqref{Eqn:Transportoperator} seems to be 
quite artificial, but it is naturally related to the singularity of the transport operator. More precisely, for a continuous variable $\xi$ one easily shows that
\begin{align}
\notag
-\at{\xi+6-n}\delta^\prime_{n-6}\at\xi=\delta_{n-6}\at\xi
\end{align}
holds in the sense of distributions, where $\delta_{n-6}\at\xi$ is the Dirac distribution supported in $\xi=n-6$. 
Our discretization of the transport operator satisfies a similar identity which can be seen by setting
$g_n^k=\eps^{-1}\delta_n^{k_n}$  in \eqref{Eqn:Transportoperator}. The Kronecker delta in 
\eqref{Eqn:Transportoperator} therefore guarantees that the resulting discrete scheme resembles the 
continuous dynamics even if mass is concentrated near the singularities.

\par
With \eqref{Eqn:Transportoperator} the discrete dynamical model reads
\begin{align}
\label{Eqn:System.N}
\frac{\dint}{\dint{t}}g_n^k-2g_n^k-\at{\xi_k+6-n}_+\nabla^+g_n^k+
\at{\xi_k+6-n}_-\nabla^-g_n^k+\delta_k^{k_n}{g_n^k}
=\Gamma\nat{Jg^k}_n\,,
\end{align}
where the coupling weight $\Gamma$ will be defined in Section \ref{Ss.mbccw}. To close the system \eqref{Eqn:System.N} we impose the boundary 
conditions
\begin{align}
\label{Dyn.BoundaryConditions}
g_n^0=0\qquad\mbox{ for  }\;7\leq{n}\leq{N},\qquad\qquad
g_n^K=0\qquad\mbox{ for  }\;2\leq{n}\leq{N},
\end{align}
so \eqref{Eqn:System.N} becomes an evolution equation for the variables $g_n^k$ with
$n=2\tdots{N}$ and $k=1\tdots{K-1}$. Notice that the boundary conditions for $k=0$ stem naturally from 
\eqref{boundarycond}, whereas those for $k=K$ reflect  the cut off in $\xi$.
%
%
%
\subsection{Moment balances and choice of the coupling weight}\label{Ss.mbccw}
%
%
In complete analogy to the discussion in Section \ref{S.intro} we 
 choose the discrete coupling coefficient $\Gamma$ such that
\eqref{Eqn:System.N} with \eqref{Dyn.BoundaryConditions} conserves the area. In order to
identify the correct formula we start with an auxiliary result for a discrete
moment $Z_n$ with
\begin{align*}
Z_n:=\eps\sum_{k=1}^{K-1}\mu_n^kg_n^k,
\end{align*}
where $\mu_n^k$ are arbitrary moment coefficients.
\begin{lemma}
\label{Lem:Dyn.Abstract.Moment}
We have
\begin{align}
\label{Lem:Dyn.Abstract.Moment.Eqn1}
\frac{\dint}{\dint{t}}Z_n-2Z_n+\zeta_n=\Gamma\,\eps\sum_{k=1}^{K-1}\mu_n^k\bat{Jg^k}_n,
\end{align}
where
\begin{align}
\label{Lem:Dyn.Abstract.Moment.Eqn2}
\zeta_n:=(6-n)_+\,\mu_n^0g_n^1+\eps\sum_{k=1}^{K-1}\eta_n^kg_n^k,\qquad
\eta_n^{k}&:=\left\{
\begin{array}{lccl}
\nabla^+[\mu_n^k\at{\xi_k+6-n}]&&\text{for}&k<k_n,\\
\mu_{n}^{k_n}&&\text{for}&k=k_n,\\
\nabla^-[\mu_n^k\at{\xi_k+6-n}]&&\text{for}&k>k_n.
\end{array}
\right.
\end{align}
\end{lemma}
\begin{proof}
Multiplying \eqref{Eqn:System.N} by $\eps{\mu_n^k}$, and summing over $k=1\tdots{K-1}$, give
\eqref{Lem:Dyn.Abstract.Moment.Eqn1} with
$\zeta_n=\zeta^+_n+\zeta^-_n+\eps{\mu_n^{k_n}}{g_n^{k_n}}$ and
\begin{align*}
\zeta^+_n:=
-\eps\sum_{k=1}^{K-1}\mu_n^k\at{\xi_k+6-n}_+\nabla^+g_n^k
,\qquad
\zeta^-_n:=
\eps\sum_{k=1}^{K-1}\mu_n^k\at{\xi_k+6-n}_-\nabla^-g_n^k.
\end{align*}
We now reformulate $\zeta_n^-$ and $\zeta_n^+$ by means of the {discrete
integration by parts} formula
\begin{align}
\notag
\eps\sum_{k=K_1}^{K_2}v^k\nabla^+u^k+\eps\sum_{k=K_1}^{K_2}u^k\nabla^-v^k=
u^{K_2+1}v^{K_2}-u^{K_1}v^{K_1-1}\,, \qquad K_2\geq K_1 \geq 1\,.
\end{align}
To this end we consider the following two cases:
\par
\emph{Case I: $2\leq{n}\leq{6}$.}
By definition, we have $\xi_k+6-n\geq0$ for all $k=1\tdots{K-1}$. This implies $\zeta^-_n=0$
and hence
\begin{align*}
\zeta^+_n&=
-\eps\sum_{k=1}^{K-1}\mu_n^k\at{\xi_k+6-n}\nabla^+g_n^k
=
\at{6-n}\mu_n^0g_n^1+\eps\sum_{k=1}^{K-1}g_n^k\nabla^-[\mu_n^k\at{\xi_k+6-n}],
\end{align*}
where we used $\xi_0=0$ and the boundary condition $g_n^K=0$.
\par%
\emph{Case II: $7\leq{n}\leq{N}$.}
Here we have $1\leq k_n<K-1$ with $\xi_{k_n}+6-n=0$, and hence
\begin{align*}
\zeta^+_n&=
-\eps\sum_{k=k_n+1}^{K-1}\mu_n^k\at{\xi_k+6-n}\nabla^+g_n^k
=\eps\sum_{k=k_n+1}^{K-1}g_n^k\nabla^-[\mu_n^k\at{\xi_k+6-n}].
\end{align*}
Similarly, we find 
\begin{align*}
\zeta_n^-&=
-\eps\sum_{k=1}^{k_n-1}\mu_n^k\at{\xi_k+6-n}\nabla^-g_n^k
=
\eps\sum_{k=1}^{k_n-1}g_n^k\nabla^+[\mu_n^k\at{\xi_k+6-n}]
\end{align*}
thanks to the boundary conditions $g_n^0=g_n^K=0$.
\end{proof}
We next summarize some elementary properties of the coupling operator $J$.
\begin{lemma}
%
The coupling matrix $J$ satisfies
\begin{align*}
\sum_{n=2}^{N}\theta_n\at{Jf}_n=\sum_{n=2}^{N-1}\at{\theta_{n+1}-\theta_n}\beta n f_n-
\sum_{n=3}^N\at{\theta_n-\theta_{n-1}}(\beta+1)nf_n,
\end{align*}
where $\theta_n$ denote arbitrary weights. In particular, we have
\begin{align}
\label{Rem:PropertiesCouplingOp.Eqn2}
\sum_{n=2}^{N}\at{Jf}_n=0,\qquad
\sum_{n=2}^{N}\at{6-n}\at{Jf}_n=
\beta N f_N- \beta 2 f_2+\sum_{n=3}^{N}nf_n.
\end{align}
\end{lemma}
\begin{proof}
The definitions from \eqref{intro.DefJ} imply
\begin{align*}
\sum_{n=2}^N\theta_n\at{Jf}_n =&
\sum_{n=2}^{N-1}\theta_n(\beta+1)(n+1)f_{n+1}
+\sum_{n=3}^{N}\theta_n \beta (n-1)f_{n-1}\\
&
-\sum_{n=2}^{N-1}\theta_n \beta n f_n
-\sum_{n=3}^{N}\theta_n(\beta+1)nf_n,
\end{align*}
and all claims follow immediately by direct computations.
\end{proof}
For the following considerations we introduce the discrete moments
\begin{align*}
X_n:=\eps\sum_{k=1}^{K-1}g_n^k,\qquad
Y_n:=\eps\sum_{k=1}^{K-1}\xi_kg_n^k\,
\end{align*}
as well as the auxiliary quantity
\begin{align*}
Q:=\eps\sum_{n=2}^N\sum_{k=1}^{K-1}\sgn\at{\xi_k+6-n}g_n^k.
\end{align*}
We also define the discrete area and the discrete polyhedral defect by
\begin{align*}
A:=\sum_{n=2}^NY_n,\qquad
P:=\sum_{n=2}^N\at{n-6}X_n.
\end{align*}
\begin{corollary}
\label{Dyn:Cor:XandY}
We have
\begin{align}
\label{Dyn:Cor:XandY.Eqn1}
\frac{\dint}{\dint{t}}X_n&=X_n+\Gamma\at{JX}_n-(6-n)_+\,g_n^1\,,
\\%
\label{Dyn:Cor:XandY.Eqn2}
\frac{\dint}{\dint{t}}Y_n&=\at{n-6}X_n+\Gamma\at{JY}_n+\eps^2\sum_{k=1}^{K-1}\sgn\at{\xi_k+6-n}
g_n^k\,.
\end{align}
\end{corollary}
\begin{proof}
The claim for $X_n$ follows from Lemma \ref{Lem:Dyn.Abstract.Moment} since
$\mu_n^k=1$ implies
\begin{align*}
\eta_n^{k}&=\left\{
\begin{array}{lclccl}
\nabla^+[{\xi_k+6-n}]&=&1&&\text{for}&k<k_n,\\
1&&&&\text{for}&k=k_n,\\
\nabla^-[{\xi_k+6-n}]&=&1&&\text{for}&k>k_n.
\end{array}
\right.
\end{align*}
Similarly, with $\mu_n^k=\xi_k$ we find 
\begin{align*}
\eta_n^{k}&=\left\{
\begin{array}{lclccl}
\nabla^+[\xi_k\at{\xi_k+6-n}]&=&2\xi_k+{6-n}+\eps&&\text{for}&k<k_n,\\
\xi_k&=&2\xi_k+{6-n}&&\text{for}&k=k_n,\\
\nabla^-[\xi_k\at{\xi_k+6-n}]&=&2\xi_k+{6-n}-\eps&&\text{for}&k>k_n.
\end{array}
\right.
\end{align*}
This means $\eta_n^k=2\xi_k+{6-n}-\eps\sgn\at{\xi_k+6-n}$ for all $n=2\tdots{N}$ and
$k=1\tdots{K-1}$, so Lemma \ref{Lem:Dyn.Abstract.Moment} implies
\eqref{Dyn:Cor:XandY.Eqn2}.
\end{proof}
We now show that the initial value problem for the differential system \eqref{Eqn:System.N} has a global unique
solution with state space
\begin{align*}
\calU=\left\{%
\at{g_n^k}_{n=2\tdots{N},\,k=1\tdots{K-1}}
\;:\;%
A=1,\;
P+\eps{Q}=0
\right\},
\end{align*}
provided that $\Gamma=\Gamma_\nom/\Gamma_\denom$ is given by
\begin{align}
\label{Eqn:Def.MuAndGamma}
\begin{split}
\Gamma_\nom&:=\sum_{n=2}^5\at{6-n}\at{
6-n+\eps} g_n^1,
\\%
\Gamma_\denom&:=\sum_{n=2}^N\at{6-n}\at{JX}_n-\eps^2
\sum_{n=2}^{N}\sum_{k=1}^{K-1}\sgn\at{\xi_k+6-n}\bat{Jg^{k}}_n.
\end{split}
\end{align}
\begin{lemma}
%
Definition
\eqref{Eqn:Def.MuAndGamma} implies that $\calU$ is invariant under the flow of
\eqref{Eqn:System.N} with \eqref{Dyn.BoundaryConditions}.
\end{lemma}
\begin{proof}
From  Corollary \ref{Dyn:Cor:XandY} and \eqref{Rem:PropertiesCouplingOp.Eqn2}$_1$ we infer that
\begin{align}
\label{App:EvolP.Eqn1}
\frac{\dint}{\dint{t}}A=P+\eps{Q},\qquad
\frac{\dint}{\dint{t}}P-P=
\sum_{n=2}^{5}\at{6-n}^2g_n^1-
\Gamma\sum_{n=2}^N\at{6-n}\at{JX}_n.
\end{align}
We now compute $\dint{Q}/\dint{t}$  using Lemma
\ref{Lem:Dyn.Abstract.Moment}. With $\mu_n^k=\sgn\at{\xi_k+6-n}$ we find
\begin{align*}
\eta_n^{k}&=\left\{
\begin{array}{lclccl}
-\nabla^+[{\xi_k+6-n}]&=&-1&&\text{for}&k<k_n,\\
0&&&&\text{for}&k=k_n,\\
+\nabla^-[{\xi_k+6-n}]&=&1&&\text{for}&k>k_n
\end{array}
\right.
\end{align*}
and hence $\mu_n^k=\eta_n^k$. Using \eqref{Lem:Dyn.Abstract.Moment.Eqn1} and
\eqref{Lem:Dyn.Abstract.Moment.Eqn2} we therefore conclude that
\begin{align}
\notag
\frac{\dint}{\dint{t}}Q-Q=-\sum_{n=2}^{5}\at{6-n}g_n^{1}+
\eps \Gamma\sum_{n=2}^{N}\sum_{k=1}^{K-1}\sgn\at{\xi_k+6-n}\bat{Jg^{k}}_n.
\end{align}
Combining this with \eqref{App:EvolP.Eqn1} we get
\begin{align*}
\frac{\dint^2}{\dint{t}^2}A=\frac{\dint}{\dint{t}}\at{P+\eps{Q}}-\at{P+\eps{Q}}=
\Gamma_\nom-\Gamma\cdot\Gamma_\denom,
\end{align*}
where $\Gamma_\nom$ and $\Gamma_\denom$ are defined in \eqref
{Eqn:Def.MuAndGamma}. In
particular,  $\Gamma=\Gamma_\nom/\Gamma_\denom$ implies
$\dint(P+\eps{Q})/\dint{t}=\dint{A}/\dint{t}=0$ for all states from $\calU$.
\end{proof}
%
%
%
\subsection{Global solutions and steady states}
%
Our next goal is to establish the existence of global in time solutions to the discrete dynamical system~\eqref{Eqn:System.N}. In particular, we show that
the restriction $\beta <2$ implies that
the denominator of $\Gamma$ is strictly positive for all times.
\begin{lemma}
\label{Dyn.Rem1}
There exist  constants $D_N$ and $\eps_N<1$ that depend only on $N$ such that for all states in 
\begin{align*}
\calU_{+}=\left\{%
\bat{g_n^k}\in\calU
\;:\;%
g_n^k\geq0\quad\mbox{ for  }\quad
n=2\tdots{N},\,k=1\tdots{K-1}
\right\}\,,
\end{align*} 
we have
\begin{align}
\label{Dyn.Rem1.Eqn1} %
(2-\beta)\sum_{n=2}^NX_n\leq\Gamma_\denom\leq{D_N}\sum_{n=2}^NX_n
\end{align}
and
\begin{align}
\label{Dyn.Rem1.Eqn2} %
\frac{2-\beta}{L}\leq\Gamma_\denom\leq\frac{D_N}{\eps},\qquad
\Gamma_\nom\leq\frac{16+4\eps}{\eps^2}\,,
\end{align}
provided that $0<\eps\leq\eps_N$. 
\end{lemma}
\begin{proof}
For all states from ${\calU}_+$ we find, using \eqref{Rem:PropertiesCouplingOp.Eqn2}$_2$ and
\eqref{Eqn:Def.MuAndGamma}, that
\begin{align}
\label{Eqn:Dyn.Fradkov.Gamma}
\Gamma_\denom=-2(\beta+1)X_2+\beta N X_N+6\sum_{n=2}^NX_n
-\eps^2
\sum_{n=2}^{N}\sum_{k=1}^{K-1}\sgn\at{\xi_k+6-n}\bat{{g}^k+Jg^{k}}_n,
\end{align}
where we used that $\sum_{n=2}^NnX_n=6\sum_{n=2}^NX_n-\eps{Q}$ thanks to $P+\eps{Q}=0$. With \eqref{Eqn:Dyn.Fradkov.Gamma} and
\begin{align*}
\abs{\eps^2\sum_{k=1}^{K-1}\sgn\at{\xi_k+6-n}\bat{{g}^k+Jg^k}_n}
\leq\eps^2{D_N}\sum_{k=1}^{K-1}g_n^{k}=\eps{D_N}{X_n}
\end{align*}
we find
\begin{align*}
\at{4-2\beta-\eps{D_N}}\sum_{n=2}^NX_n\leq\Gamma_\denom\leq\at{\beta N +6+\eps
{D_N}}\sum_{
n=2 } ^NX_n,
\end{align*}
which then implies \eqref{Dyn.Rem1.Eqn1} for all sufficiently small $\eps$. In view of
$\eps\leq\xi_k\leq{L}-\eps$ for all $k=1\tdots{K-1}$ we have that
$\eps X_n \leq Y_n \leq L X_n$, $2 \leq n \leq N$, and conclude that
\begin{align*}
\eps\sum_{n=2}^NX_n\leq 1 \leq{L}\sum_{n=2}^NX_n,
\end{align*}
so \eqref{Dyn.Rem1.Eqn1} implies \eqref{Dyn.Rem1.Eqn2}$_1$. Moreover,
\eqref{Dyn.Rem1.Eqn2}$_2$ holds because we have $\Gamma_\nom\leq\at{16+4\eps}\sum_{n=2}^5g_n^1$ 
and
$\sum_{n=2}^Ng_n^1\leq\eps^{-2}$.
\end{proof}
Now we are able to prove that the initial value problem for the discrete model is globally 
well-posed with state space ${\cal U}_{+}$.
\begin{lemma}
Let $0<\eps\leq\eps_N$. Then,  for any initial data from $\calU_{+}$  there exists a unique global
solution to \eqref{Eqn:System.N}-\eqref{Dyn.BoundaryConditions} that takes values in $\calU_+$ for all times $t\geq0$.
\end{lemma}
\begin{proof}
The estimates from Lemma \ref{Dyn.Rem1} guarantee that the mapping 
$g\in\calU\mapsto\Gamma$ is locally Lipschitz, so local existence and uniqueness of
a solution with values in $\calU$ follow from standard results. Moreover, due to the upwind discretization of the transport operator we easily show that the flow preserves the nonnegativity of $g$. Finally, conservation of area and the estimates from Lemma \ref{Dyn.Rem1} imply that $\Ga$ is uniformly bounded in time, and hence 
the global existence of
solutions.  
\end{proof}
Since the set $\calU_+$ is convex and compact, the existence of steady state solutions follows from standard results.
\begin{corollary}
For all sufficiently small $\eps$ there exists a steady state solution $g\in\calU_{+}$ to \eqref{Eqn:System.N}-\eqref{Dyn.BoundaryConditions}.
\end{corollary}
\begin{proof}
See, for instance, Proposition 22.13 in  \cite{Am90}.
\end{proof}

\begin{remark}\label{Rem.homogeneity}
If $g$ a steady state solution to \eqref{Eqn:System.N} and \eqref{Dyn.BoundaryConditions}  then
so is $\lambda g$ for any $\lambda >0$ with the same coupling weight $\Gamma$. 
\end{remark}

We conclude with further properties of steady state solutions.

\begin{lemma}
\label{Limit:Lemma1}
Each steady state solution $g$ to \eqref{Eqn:System.N} and \eqref{Dyn.BoundaryConditions} satisfies
\begin{align}
\label{Limit:Lemma1.PEqn1}
\sum_{n=2}^5\at{6-n}g_n^1=
\sum_{n=2}^NX_n,
\end{align}
and thus we have
\begin{math}
\displaystyle\Gamma\leq\frac{5}{2-\beta}
\end{math} %
for all  $\eps\leq\eps_N$.
\end{lemma}
\begin{proof}
Equation \eqref{Limit:Lemma1.PEqn1} follows from summing over $n=2\tdots{N}$ in the stationary version
of \eqref{Dyn:Cor:XandY.Eqn1} with the help of \eqref{Rem:PropertiesCouplingOp.Eqn2}. 
Using \eqref{Eqn:Def.MuAndGamma} and $\eps_N<1$ we then derive
\begin{align*}
\Gamma_\nom\leq5\sum_{n=2}^5 \at{n-6}g_n^1=
5\sum_{n=2}^NX_n,
\end{align*}
and combining this with \eqref{Dyn.Rem1.Eqn1} we find the desired result.
\end{proof}

%
%
\section{Existence of self-similar solutions}
%
In this section we study the steady states of \eqref{Eqn:System.N} and \eqref{Dyn.BoundaryConditions}  
 for fixed $N$ and $L>N-6$, and pass to the limit $\eps\to0$. We thus obtain self-similar profiles to the Fradkov model with $N<\infty$ that 
turn out to have compact support in
 $\ccinterval{0}{N-6}$. Afterwards we show that these self-similar profiles converge as $N\to\infty$.
\subsection{Limit $\eps\to 0$}\label{S.Limit1}
%
 For each $\eps=L/K$ we  choose a steady state $g$  to 
\eqref{Eqn:System.N} and \eqref{Dyn.BoundaryConditions} with coupling weight $\Gamma=\Gamma^{\eps}$,
 from which we construct a piecewise continuous function $g_{n}^\eps$ in $\ccinterval{0}{L}$ via 
\begin{align}
\label{Eqn:LimitIndentification}
g_n^\eps\at{\xi_k+\xi}=g_n^k
\qquad\qquad\mbox{ for  }\quad
n=2...N,\quad
k=1...K-1,\quad
\abs{\xi}<\tfrac12\eps.
\end{align}
In consistency with the boundary conditions \eqref{Dyn.BoundaryConditions} we further define 
\begin{align*}
g_n^\eps\at{\xi}=0\qquad\qquad\mbox{ for  }\quad
n=7\tdots{N},\quad0<\xi<\tfrac{1}{2}\eps,
\end{align*}
as well as
\begin{align*}
g_n^\eps\at{L-\xi}=0\qquad\qquad\mbox{ for  }\quad
n=2\tdots{N},\quad
0<\xi<\tfrac{1}{2}\eps.
\end{align*}
To ensure that the functions $g_n^\eps$ are well-defined on $\ccinterval{0}{L}$ we also set
\begin{align*}
g_6^\eps\at{\xi}=0\qquad\qquad\mbox{ for  }\quad
0<\xi<\tfrac{1}{2}\eps,
\end{align*}
but require continuity at $\xi=0$ for $n=2\tdots{5}$, that means
\begin{align*}
g_n^\eps\at{\xi}=g_n^\eps\at{\eps}
\qquad\qquad\mbox{ for  }\quad
n=2\tdots{5},\quad
0<\xi<\tfrac{1}{2}\eps.
\end{align*}
In consistency with \eqref{Eqn:LimitIndentification} we furthermore write
\begin{align*}
A^\eps=\sum_{n=2}^NY^\eps_n
,\qquad%
P^\eps=\sum_{n=2}^N\at{n-6}Y^\eps_n,
\end{align*}
where
\begin{align*}
X_n^\eps=\int\limits_{\eps/2}^{N-6-\eps/2}g_n^\eps\at{\xi}\dint\xi,\qquad
Y_n^\eps=\int\limits_{\eps/2}^{N-6-\eps/2}\xi g_n^\eps\at{\xi}\dint\xi.
\end{align*}
This implies 
\begin{align*}
\int\limits_{0}^{N-6}g_n^\eps\at{\xi}\dint\xi=X_n^\eps
\quad\mbox{ for  }\;n=6\tdots{N}
\qquad\text{but}\qquad%
\int\limits_{0}^{N-6}g_n^\eps\at{\xi}\dint\xi=X_n^\eps+\frac{\eps}{2}g_n^\eps\at{0}
\quad\mbox{ for  }\;n=2\tdots{5}.
\end{align*}
For the following consideration it is convenient to drop the condition $A^\eps=1$ and to scale the steady state solutions differently. Specifically,
recalling Remark \ref{Rem.homogeneity}, and due to \eqref{Limit:Lemma1.PEqn1},  we can assume that
\begin{align}
\label{Eqn:Discrete.Bounds.1}
\sum_{n=2}^5\at{6-n}g_n^\eps\at{0}=\sum_{n=2}^N X_n^\eps=1.
\end{align}
This normalization gives rise to the following uniform 
$\fspace{BV}$-estimates.
\begin{lemma}
\label{Limit:Lemma2}
For each $0<\delta<1$ and $M>6$  there exists a constant $C_{\delta, M}$ 
that is independent of $N$, $K$, and $L$ 
such that
\begin{align*}
\sum_{n=2}^5\int\limits_0^{L}\abs{\partial_{\xi}g_n^\eps}\dint\xi+
\int\limits_\delta^{L}\abs{\partial_{\xi}g_6^\eps}\dint\xi+
\sum_{n=7}^{M}\at{\int\limits_0^{n-6-\delta}\abs{\partial_{\xi}g_n^\eps}\dint\xi+
\int\limits_{n-6+\delta}^{L}\abs{\partial_{\xi}g_n^\eps}\dint\xi}\leq{C_{\delta,M}},
\end{align*}
holds for all $N\geq{M}$ and $0<\eps<\delta$, where  the measure $\abs{\partial_{\xi}g_n^\eps}\dint\xi$ 
denotes the total variation of $g_n^\eps$.
\end{lemma}
\begin{proof}
The assertion follows from the stationary version
of \eqref{Eqn:System.N} since \eqref{Eqn:Discrete.Bounds.1} provides uniform $\fspaceL^1$ bounds 
for $g_n^\eps$ and because $\Gamma^\eps$ is uniformly bounded from above, see Lemma \ref{Limit:Lemma1}.
\end{proof}
From Lemma \ref{Limit:Lemma1}, Lemma \ref{Limit:Lemma2}, and the normalization condition \eqref{Eqn:Discrete.Bounds.1} we now infer that there exists a
subsequence $\eps\to0$ such that 
\begin{align}
\label{Eqn:Limit.Convergence1}
\Gamma^\eps\quad\xrightarrow{\eps\to0}\quad\Gamma,\qquad\qquad
g_n^\eps\at{0}\quad\xrightarrow{\eps\to0}\quad\bar{g}_n\quad\text{for}\;n=2\tdots{5},
\end{align}
and
\begin{align*}
\begin{array}{lclcl}
g_n^\eps&\quad\xrightharpoonup{\eps\to0}\quad&g_n&\text{for}&n=2\tdots{5},\\
g_n^\eps&\quad\xrightharpoonup{\eps\to0}\quad&g_n+m_n\delta_{n-6}&\text{for}&n=6\tdots{N},\\
\end{array}
\end{align*}
weakly-$\star$ in the space of Radon measures $\fspace{M}\bat{\ccinterval{0}{L}}$. Here $\delta_{n-6}$ denotes the
delta distribution in $\xi=n-6$, $m_6\tdots{m_N}$ are some nonnegative numbers, and each
function $g_n$ is nonnegative and integrable in $[0,L]$.
It readily follows from \eqref{Eqn:Discrete.Bounds.1} and \eqref{Eqn:Limit.Convergence1} that 
\begin{equation}\label{Eq:24b}
 \sum_{n=2}^5 (6-n) {\bar g}_n = 1\,.
\end{equation}

\bigpar
We now exploit the weak formulation  of the stationary version of 
\eqref{Eqn:System.N} and show that the functions $g_n$ satisfy -- outside the set of possible singularities --
 the ordinary differential equation \eqref{ss1} for self-similar profiles. Moreover, using the weak formulation we also 
recover the boundary conditions and derive algebraic relations for the possible singularities.
\begin{lemma}
\label{Lem:WeakFormulation.A}
For each $n=2\tdots{N}$, the function $g_n$ satisfies
\begin{align}
\label{Lem:WeakFormulation.Proof1}
\begin{split}
\int\limits_0^{L}g_n\bat{\at{\xi+6-n}\phi'-\phi}\dint\xi+(6-n)_+\bar{g}_n\phi\at{0}-
\chi_{n\geq6}m_n\phi\at{n-6}\\=
\Ga\at{\int\limits_0^{L}\at{Jg}_n\phi\dint\xi+\at{J\omega(\phi)}_n}
\end{split}
\end{align}
for all smooth test functions $\phi$, where $\omega_n(\phi):=\chi_{n\geq6}m_n\phi\at{n-6}$. 
Moreover, 
we have
\begin{enumerate}
\item $g_n\in\fspaceC\at{I_n}$,
\item $g_n$ has left and right limits
at $\xi=0$, $\xi=n-7$, $\xi=n-5$ and $\xi=L$,
\item $g_n\in\fspaceC^1\nat{\hat{I}_n}$,
\item $g_n$ satisfies \eqref{ss1}  pointwise in $\hat{I}_n$,
\end{enumerate}
where $\hat{I}_n:=\oointerval{0}{L}\setminus\{n-7,n-6,n-5\}$ and
$\hat{I}_n:=\oointerval{0}{L}\setminus\{n-8,n-7,n-6,n-5,n-4\}$. 
\end{lemma}
\begin{proof}
We employ Lemma \ref{Lem:Dyn.Abstract.Moment} as follows. For a given smooth test function $\phi:\Rset\to\Rset$ we set
\begin{align*}
\mu_n^k=\mu^k=\eps^{-1}\int\limits_{\xi_k-\tfrac{1}{2}\eps}^{\xi_k+\tfrac{1}{2}\eps}\phi\at\xi\dint\xi=\phi\at{
\xi_k }+\DO{\eps^2},
\end{align*}
which gives
\begin{align*}
\eta_n^{k}&=\frac{1}{\eps} \int\limits_{\xi_k-\tfrac{1}{2}\eps}^{\xi_k+\tfrac{1}{2}\eps}\phi(\xi)+ (\xi+6-n)\phi'(\xi)\,d\xi +O(\eps)
\qquad\quad\mbox{ for  }\quad
k=1\tdots{K-1},
\end{align*}
where all error terms depend only  on $\phi$ and $N$. Using \eqref{Lem:Dyn.Abstract.Moment.Eqn1}, 
and thanks to our definition of $g_n^{\eps}\at{\xi}$ and $g_n^{\eps}\at{L-\xi}$ for $0\leq\xi\leq\eps/2$, we therefore find
\begin{align*}
\int\limits_{0}^{L}g_n^\eps\Bat{\at{\xi+6-n}\phi'-\phi}
\dint\xi+(6-n)_+\, g_n^\eps\at{0}\phi\at{0}=\Gamma^\eps
\int\limits_{0}^{L}\at{Jg^\eps}_n\phi\dint\xi+\DO{\eps}.
\end{align*}
The limit $\eps\to0$ now yields \eqref{Lem:WeakFormulation.Proof1}. 
\par
Since we have $\phi\at{0}=\omega_n(\phi)=0$ and $J\at{\omega(\phi)}_n=0$
for each $\phi$ with compact support in ${I}_n$,  we find
\begin{equation}\label{Eq:z1}
 \big( (\xi+6-n) g_n(\xi)\big)'= g_n(\xi) + \Gamma \big( Jg\big)_n(\xi) \qquad \mbox{ in } \quad{\cal D}^\prime\nat{{I}_n}\,.
\end{equation}
Since the right-hand side of \eqref{Eq:z1} is integrable in $I_n$, we conclude that the function $h_n$ defined by
$h_n(\xi) = (\xi+6-n)g_n(\xi)$ for $\xi \in\oointerval{0}{L}$ belongs to $\fspaceW^{1,1}\nat{I_n}\subset\fspace{C}\nat{I_n}$, and this implies that
$h_n\in$ posesses well-defined one-sided limits at $\ol{I}_n\setminus{I}_n$. Consequently, 
$g_n$ is continuous in $I_n$ and has well defined one-sided limits at all points $\xi\in\{0,n-7,n-5,L\}$.
Finally, using again \eqref{Eq:z1} we deduce that $g_n$ is continuously differentiable in $\hat{I}_n$ and  satisfies \eqref{ss1} pointwise in that set.
\end{proof}

\begin{lemma}
\label{Lem:WeakFormulation.B}
 The following assertions are satisfied:
\begin{enumerate}
\item
We have
\begin{align}
\label{Lem:WeakFormulation.Eqn2}
\at{\Gamma\kappa_n-1}m_n=0\qquad\mbox{ for  }\quad
n=6\tdots{N},
\end{align}
where $\kappa_n$ denotes the modulus of the $n^\text{th}$ diagonal element of the coupling matrix $J$, that means
\begin{align*}
\kappa_n:=\left\{
\begin{array}{lcl}
2\beta&\text{for}&n=2,\\
(2\beta+1)n&\text{for}&n=3\tdots{N-1},\\
(\beta+1)N&\text{for}&n=N.
\end{array}
\right.
\end{align*}
\item
\begin{enumerate}
\item
$\lim_{\xi\searrow0}g_n\at{0}=\bar{g}_n$ for $n=2\tdots{4}$,
\item
$\lim_{\xi\searrow0}g_5\at{\xi}=\bar{g}_5-6 \Gamma (\beta+1)m_{6}$,
\item
$\lim_{\xi\searrow0}g_7\at{\xi}=6 \Gamma  \beta m_{6}$,
\item
$\lim_{\xi\searrow0}g_n\at{\xi}=0$ for $n=8\tdots{N}$.
\end{enumerate}
Moreover, $\lim_{\xi\nearrow{L}}g_n\at{\xi}=0$ for all $n=2\tdots{N}$.
\item
We have
\begin{align}
\label{Lem:WeakFormulation.Eqn3}%
\jump{g_n}\at{n-7}=\Gamma  \beta (n-1) m_{n-1}
\qquad\mbox{ for  }\quad
n=8\tdots{N}
\end{align}
and
\begin{align}
\label{Lem:WeakFormulation.Eqn4}%
-\jump{g_n}\at{n-5}=\Gamma(\beta+1)(n+1)m_{n+1}
\qquad\mbox{ for  }\quad
n=6\tdots{N},
\end{align}
where $\jump{g_n}\at{\xi}:=\lim_{s\searrow0}\bat{g_n\at{\xi+s}-g_n\at{\xi-s}}$.
\end{enumerate}
\end{lemma}
\begin{proof}
\emph{1.} 
Let $n=6\tdots{N}$. Since $g_n$ is integrable in $(0,L)$, we have
\[\lim_{s\searrow0}(n-6\pm s) g_n (n-6 \pm s) =0
 \]
 because otherwise $g_n$ would not be integrable. Integrating by parts in \eqref{Lem:WeakFormulation.Proof1}, and using \eqref{ss1}, we therefore find
\begin{align*}
-m_n\phi\at{n-6}=-\Gamma\kappa_nm_n\phi\at{n-6}
\end{align*}
for all test functions $\phi$ that have support in $\oointerval{n-7}{n-5}$. Thus we have shown
\eqref{Lem:WeakFormulation.Eqn2}.
\par \emph{2.} 
Let $\phi$ be an arbitrary test function with support in $\oointerval{-1}{1}$.
Combining \eqref{Lem:WeakFormulation.Proof1} with \eqref{ss1} implies, again employing integration by parts and the continuity properties of $g_n$, that
\begin{align*}
\at{-\lim_{\xi\searrow0} \xi g_n(\xi)+(6-n)_+\,\bar{g}_n-\delta_n^6 m_n}\phi\at{0}=\\
\Gamma\Bat{6 (\beta+1)\delta_n^5 - \kappa_6 \delta_n^6 +6 \beta \delta_n^7} m_6\phi\at{0}.
\end{align*}
From this identity we readily derive the claimed formulas for $\lim_{\xi\searrow0}g_n\at\xi$. Moreover, considering test functions $\phi$ with support in $\oointerval{L-1}{L+1}$ we find $\lim_{\xi\nearrow L}g_n\at\xi=0$ for all $n=2\tdots{N}$.

\par \emph{3.}
Now let $n\geq8$ and suppose that $\phi$ is supported in $\oointerval{n-8}{n-6}$. From
\eqref{Lem:WeakFormulation.Proof1} and \eqref{ss1} we now derive
\begin{align*}
-\jump{g_n}\at{n-7}\phi\at{n-7}=\Gamma \beta(n-1) m_{n-1}\phi\at{n-7},
\end{align*}
which implies \eqref{Lem:WeakFormulation.Eqn3}. The proof of
\eqref{Lem:WeakFormulation.Eqn4} is analogous.
\end{proof}
As an easy consequence we obtain
that all functions $g_n$ vanish for $\xi>N-6$.
\begin{lemma}
We have $g_n\at\xi=0$ for all $\xi\in\ocinterval{N-6}{L}$ and $n=2\tdots{N}$.
\end{lemma}
\begin{proof}
Standard results from the theory of ordinary differential equations imply that the initial value problem for the system
\eqref{ss1} with prescribed data at $\xi=L$ is well-posed on the
interval $\ocinterval{N-6}{L}$. The claim therefore follows from the boundary conditions for
$\xi=L$, see Lemma \ref{Lem:WeakFormulation.B}.
\end{proof}

From \eqref{Lem:WeakFormulation.Eqn2} we conclude that at most one of the weights $m_n$ of the Dirac masses
does not vanish. However,  in order to show that all weights vanish we need a 
better understanding of the properties of $g_{n}$. 
%
%
\subsection{Self-similar solutions for $N<\infty$}\label{S.Limit2}
%
In this section we characterise the properties of the functions $g_n$ 
in more detail, in particular the behaviour near $\xi=n-6$. The results allow us to conclude that all 
weights $m_n$ must vanish and that the functions $g_n$ therefore provide in fact a self-similar solution to the Fradkov model.
\par%
All subsequent considerations rely on the solution formula
\begin{align}
\label{Eqn:SolutionFormula}
g_n\at{n-6\pm{s}}=s^{-2+\Gamma\kappa_n}\at{t^{2-\Gamma\kappa_n}g_n\at{n-6\pm{t}}+
\int\limits_{s}^{t}y^{1-\Gamma\kappa_n}{G_n}\at{n-6\pm{y}}\dint{y}},
\end{align}
which is direct consequence of Lemma \ref{Lem:WeakFormulation.A} and the Variation of Constants Principle. Here,
\begin{align}
\label{Eqn:SolutionFormulaAux}
G_n:=\Gamma\,\Bat{\beta (n-1)\bat{g_{n-1}+m_{n-1}\delta_{n-7}}+
\at{\beta+1}\at{n+1}\bat{g_{n+1}+m_{n+1}\delta_{n-5}}},
\end{align}
where we set $g_1\equiv{g_{N+1}}\equiv0$ and $m_1=\tdots=m_5=m_{N+1}=0$ to simply the
notation. Notice that \eqref{Eqn:SolutionFormula} holds for all $n=2\tdots N$ and  $0<s<t$ provided that all terms are well defined, i.e., as long as $n-6\pm{s}$ and $n-6\pm{t}$ belong to
 $\ccinterval{0}{L}\setminus\{n-7,n-6,n-5\}$.
\bigpar%
In order to show that the functions $g_n$ are positive almost everywhere on $\ccinterval
{0}{N-6}$, we formulate the following auxiliary result.
\begin{lemma}
\label{Lem:PositivityResult}
Suppose there exist $2\leq m \leq{N}$ and some 
$\bar{\xi}\in\oointerval{0}{N-6}\setminus\Nset$ such that $g_{m}\nat{\bar\xi}=0$. Then we have
$g_n\at\xi=0$ for all $n=2\tdots{N}$ and all $\xi\in\oointerval{0}{N-6}\setminus\Nset$.
\end{lemma}
\begin{proof}
Due to $g_{m} \geq 0$,  the point $\bar \xi$ is minimizer of $g_m$ and thus we have
$g_{m}\nat{\bar\xi}=g_{m}'\nat{\bar\xi}=0$. Since $g_m$ solves \eqref{ss1} pointwise at $\bar\xi$, see Lemma
\ref{Lem:WeakFormulation.A}, we also find
\begin{align*}
g_{m-1}\at{\bar\xi}=g_{m+1}\at{\bar\xi}=0.
\end{align*}
Iterating this argument with respect to $n$ we finally get $g_{n}\at{\bar\xi}=0$ for
all $n=2\tdots{N}$. In particular, we have proven the implication
\begin{align}
\label{Lem:PositivityResult.PEqn1}
g_n\at{\bar\xi}=0\;\text{ for some }\;n \geq 2\qquad\implies\qquad
g_n\at{\bar\xi}=0\;\text{ for all }\; n\geq 2.
\end{align}
Since \eqref{Eqn:SolutionFormula} and $G_2\geq0$ imply
\begin{align*}
0=g_2\at{\bar\xi}\geq\at{\bar{\xi}+4}^{-2+\Gamma\kappa_2}
\at{\xi+4}^{2-\Gamma\kappa_2}g_2\at\xi\qquad\text{for}\quad\bar\xi\leq\xi\leq{N-6}\,,
\end{align*}
we conclude that $g_2\at{\xi}=0$ for 
$\xi\in\cointerval{\bar\xi}{N-6}$. Combining this with \eqref{Lem:PositivityResult.PEqn1} we
then conclude that $g_n\at{\xi}=0$ for  all
$\xi\in\cointerval{\bar\xi}{N-6}\setminus\Nset$ and $n=2\tdots{N}$.
\par We now choose an index $\hat{n}\in\Nset$ with $7\leq\hat{n}\leq{N-6}$ such that
$\bar{\xi}<\hat{n}-6$. The solution formula \eqref{Eqn:SolutionFormula} now
implies 
\begin{align*}
0=g_{\hat{n}}\at{\bar{\xi}}\geq\at{\hat{n}-6-\bar{\xi}}^{-2+\Gamma\kappa_{\hat{n}}}
\at{\hat{n}-6-\xi}^{2-\Gamma\kappa_{\hat{n}}}g_{\hat{n}}\at{\xi}\qquad\text{for}\quad 0\leq\xi\leq\bar\xi\,,
\end{align*}
and arguing as before we derive $g_n\at{\xi}=0$ for all 
$\xi\in\ocinterval{0}{\bar\xi}\setminus\Nset$ and
$n=2\tdots{N}$.
\end{proof}
\begin{lemma}
\label{Lem:Positivity}
For each $n=6\tdots{N}$ we have
$\Gamma\kappa_n>1$ and $m_n=0$. Consequently, 
for each $n=2\tdots{N}$ the function
$g_n$ is positive and continuous on 
$\oointerval{0}{N-6}\setminus\{n-6\}$, and
continuously differentiable on $
\oointerval{0}{N-6}\setminus\{n-7,n-6,n-5\}$.
\end{lemma}
\begin{proof}
Assume for contradiction that there are $m\geq 2$ and $\bar \xi \in (0,N-6) \setminus \Nset$
such that $g_m(\bar \xi)=0$. According to 
Lemma \ref{Lem:PositivityResult} 
we have  $g_n\at\xi=0$ for all $n=2\tdots{N}$ and
$\xi\in\oointerval{0}{N-6}\setminus\Nset$. 
It then follows from Lemma \ref{Lem:WeakFormulation.B} that $m_6=0$ and ${\bar g}_n=0$ for $2 \leq n \leq 5$, thereby 
contradicting the normalization condition \eqref{Eq:24b}. Consequently, we have  $g_n(\xi)>0$ for all $n\geq 2$ and $\xi \in (0,N-6) \setminus \Nset$, and the solution formula  \eqref{Eqn:SolutionFormula} ensures that
$g_n(\xi)>0$ for all $\xi \in (0,N-6) \setminus\{n-6\}$.
\par
Now let $n=7\tdots{N-1}$ be given. Since $G_n$ is nonnegative and $g_n$ positive in $\oointerval{0}{N-6}\setminus\Nset$, the solution formula \eqref{Eqn:SolutionFormula} implies that
\begin{align*}
g_n\at{n-6\pm{s}}\geq{}c_n s^{-2+\Gamma\kappa_n}\qquad\text{for all} \quad \abs{s}\leq\tfrac{1}{2},
\end{align*}
where
\begin{align*}
 c_n := \frac{\min\Big\{
g_n\at{n-6-\tfrac{1}{2}} , g_n\at{n-6+\tfrac{1}{2}}
 \Big\}}{2^{2-\Ga\kappa_n}}>0.
\end{align*}
Since 
$g_n$ is integrable, we now conclude that $\Gamma\kappa_n>1$, and
\eqref{Lem:WeakFormulation.Eqn2} yields $m_n=0$. The arguments for $m_6=0$ and $m_N=0$ are similar. 
\par
Finally, the inclusions
\begin{align*}
g_n\in\fspaceC\bat{\oointerval{0}{N-6}\setminus\{n-6\}}\,,\qquad 
g_n\in\fspaceC^1\bat{\oointerval{0}{N-6}\setminus\{n-7,n-6,n-5\}}
\end{align*}
are implied by $0=m_{n-1}=m_{n}=m_{n+1}$, see Lemma \ref{Lem:WeakFormulation.A} and Lemma \ref{Lem:WeakFormulation.B}.
\end{proof}
\begin{corollary}
\label{Cor:GammaBounds}
We have
\begin{math}
\displaystyle\frac{1}{\kappa_6}\leq\Gamma\leq\frac{5}{2-\beta}.
\end{math}
\end{corollary}
\begin{proof} 
The upper bound is provided by Lemma \ref{Limit:Lemma1}, the lower one by Lemma \ref{Lem:Positivity}.
\end{proof}
Since all weights $m_n$ vanish, we immediately arrive at the following result, which in turn implies 
that the functions $g_n$ provide indeed a self-similar profile to the Fradkov model.
\begin{corollary}
\label{Cor:PositivityImpl}
We have 
\begin{enumerate}
\item
$g_n\at{0}=\bar{g}_n>0$ for all $n=2\tdots{5}$ with $\sum_{n=2}^5g_n\at{0}=1$,
\item
$\sum_{n=2}^N{X}_n=1$ with $X_n=\int_0^{N-6}g_n\at\xi\dint\xi$,
\item 
$P=\sum_{n=2}^N\at{n-6}{X}_n=0$,
\item
$A=\sum_{n=2}^N{Y}_n>0$ with $Y_n=\int_0^{N-6}\xi\,g_n\at\xi\dint\xi$,
\item $\Gamma=\Gamma_\nom/\Gamma_\denom$ depends on $\at{X_n}_n$ and $\at{g_n\at{0}}_n$ via \eqref{Gammadef2}. 
\end{enumerate}
Moreover, the weak formulation \eqref{weakformulation} as well as the identities
\begin{align}
\label{Eqn.MomentIdentities}
\at{6-n}g_n\at{0}=X_n+\Gamma\at{JX}_n,\qquad
\at{6-n}X_n=\Gamma\at{JY}_n
\end{align}
hold for all $n=2\tdots{N}$.
\end{corollary}
We finally characterize the behaviour of $g_n$ near  $\xi=n-6$. 
\begin{lemma}
\label{Lem:NearSingularities}
For each $n=6\tdots{N}$ one of the following conditions is satisfied:
\begin{enumerate}
\item
$\Gamma\kappa_n>2$ and $g_n$ is continuous at $\xi=n-6$ with
$\lim_{\xi\to{n-6}}g_n\at{\xi}=G_n\at{n-6}/\at{\Gamma\kappa_n-2}$,
\item %
$\Gamma\kappa_n=2$ and $g_n\at{n-6\pm{s}}\sim-{G_n\at{n-6}}\ln{s}$ as $s \to 0$,
\item
$\Gamma\kappa_n<2$ and $g_n\at{n-6\pm{s}}\sim{\ell_n}s^{-2+\Gamma\kappa_n}$ for some
constant $\ell_n>0$ as $s \to 0$.
\end{enumerate}
Here $G_n$ is defined in \eqref{Eqn:SolutionFormulaAux}.
\end{lemma}
\begin{proof}
Throughout this proof we assume that $0<s<t\leq 1$. We also set
\begin{align*}
C_n\at{t}:=\sup\Big\{\babs{G_n\at{\xi}-G_n\at{n-6}}\;:\;\xi\in\ccinterval{n-6-t}{n-6+t}\Big\},
\end{align*}
and notice that $C_n\at{t}\to0$ as $t\to0$ since $G_n$ is continuous at $\xi=n-6$.
\par
\emph{Case I: $\Gamma\kappa_n>2$}. From \eqref{Eqn:SolutionFormula} we infer that
\begin{align*}
g_n\at{n-6\pm{s}}&=\at{\frac{s}{t}}^{-2+\Gamma\kappa_n}g_n\at{n-6\pm{t}}+
\frac{G_n\at{n-6}}{\Gamma\kappa_n-2}\at{1-\at{\frac{s}{t}}^{-2+\Gamma\kappa_n}}
\\&+
s^{-2+\Gamma\kappa_n}\int\limits_s^ty^{1-\Gamma\kappa_n}\Bat{G_n\at{n-6\pm{y}}-G_n\at{n-6}}
\dint{y}.
\end{align*}
We therefore find
\begin{align*}
\limsup\limits_{s\searrow0}\Babs{g_n\at{n-6\pm{s}}-G_n\at{n-6}/\at{\Gamma\kappa_n-2}}
&\leq
\frac{C_n\at{t}}{\Gamma\kappa_n-2},
\end{align*}
and the limit $t\to0$ provides the desired result.
\par
\emph{Case II: $\Gamma\kappa_n=2$}.
The solution formula \eqref{Eqn:SolutionFormula} gives
\begin{align*}
g_n\at{n-6\pm{s}}&=g_n\at{n-6\pm{t}}+G_n\at{n-6}\at{\ln{t}-\ln{s}}+
\int\limits_s^t\frac{G_n\at{n-6\pm{y}}-G_n\at{n-6}}{y}\dint{y},
\end{align*}
and due to $G_n\at{n-6}>0$, see Lemma \ref{Lem:Positivity}, we estimate
\begin{align*}
\limsup\limits_{s\searrow0}\Babs{\frac{g_n\at{n-6\pm{s}}}{G_n\at{n-6}\ln{s}}+1}
\leq%
C_n\at{t}.
\end{align*}
The claimed asymptotic behaviour now follows by letting $t\to0$.
\par
\emph{Case III: $\Gamma\kappa_n<2$}. Formula
\eqref{Eqn:SolutionFormula} implies
\begin{align*}
s^{2-\Gamma \kappa_n} g_n (n-6\pm s) =  g_n (n-6\pm 1) +
\int\limits_{s}^{1}y^{1-\Gamma\kappa_n}G_n\at{n-6\pm{y}}\dint{y},%
\end{align*}
and we conclude that the right-hand side of the above equality has a positive limit as  $s\to0$.
\end{proof}
%
%
%
\subsection{Exponential decay and estimates for higher moments}\label{S.Limit3}
%
We next prove that the moments $X_n$ decay exponentially with $n$ where the rate 
is independent of $N$. This gives rise to tightness estimates that enable us to 
pass to the limit $N \to \infty$ in Section \ref{S.Limit4}.
Introducing
\begin{align*}
z_n:=\frac{\at{n-1}X_{n-1}}{nX_{n}}, \qquad \tau:= \frac{1+\beta}{\beta}>1, \qquad
\Phi\at{z}:=1+\tau-\frac{\tau}{z}\,
\end{align*}
we readily derive from  \eqref{Eqn.MomentIdentities}$_1$ the backward recursion formula
\begin{align}
\label{Eqn.ZetaIteration}
z_N=\tau-\frac{1}{\Gamma\beta{N}}
,\qquad%
z_{n}=\Phi\at{z_{n+1}}-\frac{1}{\Gamma\beta{n}}\qquad\mbox{ for  }\quad{n}=6\tdots{N-1}
\,.%
\end{align}
Notice that $\Phi$ is strictly increasing and has exactly two fixed points  $z=1$ and $z=\tau$ with $\Phi'\at{1}>1>\Phi'\at{\tau}$.
 We therefore find that $z=1$ is unstable, whereas $z=\tau$ is stable and 
attracts all points $z>1$.
\begin{lemma}
\label{Lemma.ZetaDecay}
We have
\begin{align}
\notag
 \tau \Big ( 1- \frac{2}{n\Gamma}\Big) \leq z_n \leq \tau \qquad \mbox{ for } \quad n \geq \bar N\,,
\end{align}
where $\bar{N}$ is the smallest integer larger than $2\at{1+2\beta} \kappa_6=12\at{1+2\beta}^2$.
\end{lemma}
\begin{proof}
For each $n\in\Nset$ with $n\Ga>2$ we set $y_n:=\tau \bat{1- 2/n \Gamma}$ and find
\begin{align*}
 \Phi(y_{n+1}) - \frac{1}{\Gamma \beta n} - y_n & \geq 
 \Phi(y_{n}) - \frac{1}{\Gamma \beta n} - y_n 
 = 
\frac{1+2\beta}{\beta\Gamma{n}}
 - \frac{2}{\Gamma n - 2} \,.
\end{align*}
A direct computation reveals that the right hand side is nonnegative for $n\Ga>2\at{1+\beta}$, and 
in view of $\Ga>1/\kappa_6$ we conclude that
\begin{align*}
0\leq{y}_n\leq\Phi(y_{n+1})- \frac{1}{\Gamma \beta n}  \qquad \mbox{ for } \quad n\geq\bar{N}\,.
\end{align*}
For $N>\bar{N}$ we also have $y_N\leq{z_N}\leq\tau$, 
and using \eqref{Eqn.ZetaIteration} as well as the monotonicity of $\Phi$ we readily verify by induction that 
$y_n \leq z_n\leq\tau$ for all $n$ with $\bar N\leq{n}\leq{N}$. 
\end{proof}
\begin{corollary}
\label{Cor.xndecay}
There exist positive constants $c$ and $C$ that are independent of $N$ such that
\begin{align}
\notag
{c} \tau ^{-n} X_{\bar N }  \leq n X_n \leq C n^{2/\Ga } \tau^{-n} X_{\bar N } \qquad \mbox{ for } \quad \bar N \leq n \leq N\,.
\end{align}
\end{corollary}
\begin{proof}
By Lemma \ref{Lemma.ZetaDecay} we have 
\begin{align*}
 \tau^{n-\bar N} \geq \prod_{m=\bar N+1}^n z_m \geq  \tau^{n-\bar N} \prod_{m=\bar N+1}^n \Big( 1- \frac{2}{m \Gamma}\Big)\,.
\end{align*}
The concavity of the logarithm implies 
\begin{math}
 \ln \at {1- \frac{2}{m\Gamma}} \geq  - \frac{2}{\Gamma m- 2},
\end{math}
and hence 
\begin{align*}
 \ln \Big( \prod_{m=\bar N +1}^n \Big( 1- \frac{2}{m \Gamma}\Big) \Big) \geq
-2 \int_{\bar{N}}^n\frac{\dint{s}}{\Gamma{s}-2}
= - \frac{2}{\Gamma} 
\ln \Big( \frac{\Gamma n - 2}{\Gamma \bar N - 2}\Big)\,.
\end{align*}
We therefore find
\begin{align*}
 \tau^{n-\bar N} \geq \frac{\bar N X_{\bar N}}{n X_n} \geq  \tau^{n-\bar N}  \Big( \frac{\Gamma \bar N - 2 }{
\Gamma n - 2} \Big)^{2/\Gamma},
\end{align*}
and this implies the desired result since $\Gamma$ is bounded, see Corollary \ref{Cor:GammaBounds}. 
\end{proof}
We now exploit the exponential decay of $X_n$ and derive tightness estimates. To this end we consider the
moments
\begin{equation}
\notag
 M_{k,n}= \int_0^{N-6} \xi^k g_n(\xi)\,d\xi\,, \qquad k \geq 0\,,
\end{equation}
and notice that $X_n=M_{0,n}$ and $Y_n=M_{1,n}$.
\begin{lemma}
 \label{Lem:savethepaper}
For any $k\geq 0$ there exists a constant $C_k>0$ independent of  $N$ such that
\begin{equation}
\notag
 \sum_{n=2}^N n^k X_n + \sum_{n=2}^N M_{k,n} \leq C_k\,. 
\end{equation}
\end{lemma}

\begin{proof}
Let  $k>1$. Multiplying \eqref{ss1} by $\xi^k$, integrating over $(0,N-6)$ and using integration by parts as well as the
boundary conditions, we find 
\begin{align*}
 (k-1) M_{k,n}= k(n-6) M_{k-1,n} + \Gamma (JM_k)_n\,.
\end{align*}
Summing over $n$ and using \eqref{Rem:PropertiesCouplingOp.Eqn2}$_1$ we deduce 
\begin{align*}
  (k-1) \sum_{n=2}^N M_{k,n} = k\sum_{n=2}^N (n-6) M_{k-1,n}  \leq k\sum_{n=2}^N n M_{k-1,n} \,.
\end{align*}
Since H\"older's inequality for integrals implies
\begin{align*}
 M_{k-1,n}\leq \big( M_{k,n}\big)^{(k-1)/k} \,\big( X_n\big)^{1/k} ,
\end{align*}
we can employ H\"older's inequality for series to find
\begin{align*}
  (k-1) \sum_{n=2}^N M_{k,n}\leq k \Big( \sum_{n=2}^N M_{k,n} \Big)^{(k-1)/k} \Big( \sum_{n=2}^n n^k X_n\Big)^{1/k}\,,
\end{align*}
and hence
\begin{align*}
  \sum_{n=2}^N M_{k,n} \leq \Big( \frac{k}{k-1}\Big)^k  \sum_{n=2}^N n^k X_n \,.
\end{align*}
Thanks to Corollaries \ref{Cor:GammaBounds}, \ref{Cor:PositivityImpl}, and \ref{Cor.xndecay} we then obtain 
\begin{align*}
\sum_{n=2}^N n^k X_n \leq {\bar N}^k + C \sum_{n=\bar N+1}^N n^{k+2/\Ga -1} \tau^{-n} \leq C_k\,,
\end{align*}
and this completes the proof for $k>1$. The case $k\in (0,1]$ follows by interpolation.
\end{proof}
We finally prove that even some moments with
exponential weight are uniformly bounded.
\begin{lemma}\label{Lem:savethequeen}
 For each $0<\lambda < \ln \tau $ there exists a constant $C_{\lambda}$ that is independent of $N$ such that
\begin{align*}
  \sum_{n=2}^N e^{\lambda n} X_n + \sum_{n=2}^N \int_0^{\infty} e^{\lambda \xi} g_n(\xi) \,d\xi \leq C_{\lambda}\,.
\end{align*}
\end{lemma}

\begin{proof}
We multiply \eqref{ss1} by $e^{\lambda \xi}$ and integrate over $(0,N-6)$ to obtain
\begin{align*}
 (6-n) g_n(0) + \int_0^{N-6} e^{\lambda \xi} \big (\lambda (\xi + 6-n) -1 \big) g_n(\xi)\,d\xi= \Gamma \big( JE_{\lambda}\big)_n ,
\end{align*}
where  $E_{\lambda,n} = \int_0^{N-6} e^{\lambda \xi} g_n(\xi)\,d\xi$, and this implies
\begin{align*}
\lambda \int_0^{N-6}\xi e^{\lambda \xi} g_n(\xi)\,d\xi \leq \int_0^{N-6} n e^{\lambda \xi} g_n(\xi) \,d\xi + 
\Gamma \big( JE_{\lambda}\big)_n\,.
\end{align*}
Now we choose $\bar\la$ with $0<\bar\la<\la$ and $\la^2<\bar\la\ln\tau$, and estimate 
\begin{align*}
 \begin{split}
  \int_0^{N-6} n e^{\lambda \xi} g_n(\xi) \,d\xi &=
\int_0^{\lambda n/\bar \lambda} n e^{\lambda \xi} g_n(\xi) \,d\xi +
\int_{\lambda n/\bar \lambda}^{N-6} n e^{\lambda \xi} g_n(\xi) \,d\xi\\
& \leq n e^{\lambda^2 n/\bar \lambda} X_n + \frac{\bar \lambda}{\lambda} \int_0^{N-6} \xi e^{\lambda \xi} g_n(\xi) \,d\xi .
\end{split}
\end{align*}
This implies 
\begin{align*}
 (\lambda - \bar \lambda) \int_0^{N-6} \xi e^{\lambda \xi} g_n(\xi)\,d\xi \leq \lambda n e^{\lambda^2n/\bar \lambda} X_n + 
\Gamma \big( JE_{\lambda}\big)_n\,,
\end{align*}
and summation over $n \geq 2$ yields, thanks to \eqref{Rem:PropertiesCouplingOp.Eqn2}$_1$, 
\begin{align*}
 (\lambda - \bar \lambda) \sum_{n=2}^N \int_0^{N-6} \xi e^{\lambda \xi} g_n(\xi)\,d\xi \leq \lambda \sum_{n=2}^N e^{\lambda^2n/\bar \lambda} X_n
 \,.
\end{align*}
Finally,  the estimate
\begin{align}
\notag
 \sum_{n=2}^N e^{\lambda n} X_n + \sum_{n=2}^N e^{\lambda^2n/\bar \lambda} X_n \leq C_{\lambda}
\end{align}
follows from Corollary \ref{Cor.xndecay} due to $X_{\bar{N}}\leq\sum_{n=2}^{\bar{N}}X_n=1$ and the choice of $\la$ and $\bar\la$.
\end{proof}

%
\subsection{Limit $N\to \infty$}\label{S.Limit4}
To finish our existence proof  we construct self-similar profiles to the original Fradkov model by passing to the limit $N\to\infty$. 
 Our arguments are very similar to those used in Sections \ref{S.Limit1} and \ref{S.Limit2} for the limit $\eps\to0$, and therefore we only sketch the main ideas.
\par%
For each $N<\infty$ we define the functions $g_n^N:\cointerval{0}{\infty}\to\cointerval{0}{\infty}$ with $n\geq2$ as trivial continuation of 
the self-similar profiles constructed in  Section \ref{S.Limit2}. This means we set $g_n^N\at{\xi}=0$ for $n>N$ or $\xi>N-6$. 
We also denote the corresponding coupling weights by $\Gamma^N$ and use notations
such as $X_n^N,Y_n^N$ and $M_{k,n}^N$  to refer to the various moments of $g^N_n$.
\par
Due to the $\fspace{BV}$-estimates from Lemma \ref{Limit:Lemma2} and the normalization, see 
Corollary \ref{Cor:PositivityImpl}, there exist a subsequence $N \to \infty$, nonnegative
real numbers $({\bar g}_n)_{2\leq n\leq 5}$
and $\at{m_n}_{n\geq 6}$,  and a sequence of nonnegative integrable functions $\at{g_n}_{n\geq2}$
such that 
\begin{align*}
{g}_n^N(0) \quad\xrightarrow{N\to\infty}\quad {\bar g}_n\qquad \text{for} \quad n=2\tdots{5},
\end{align*}
and 
\begin{align*}
g_n^N\quad\xrightharpoonup{N\to\infty}\quad{g_n}+\chi_{n\geq6}m_n\delta_{n-6}
\qquad\text{weakly-$\star$ in $\fspace{M}\bat{\cointerval{0}{\infty}}$}\qquad\text{for}\quad n\geq2.
\end{align*}
Moreover, the uniform moment estimates from Lemma \ref{Lem:savethepaper} imply
$g_n \in \fspaceL^1 \bat{(0,\infty);\xi^k d\xi}$ for all $k>0$, 
and hence
\begin{align*}
1=\sum_{n=2}^{\infty}X_n^N\quad\xrightarrow{N\to\infty}\quad
\sum_{n=2}^{\infty}X_n\,,
\qquad\qquad
\sum_{n=2}^{\infty}nX_n^N\quad\xrightarrow{N\to\infty}\quad
\sum_{n=2}^{\infty}nX_n<\infty\,
\end{align*}
with $X_n=m_n + \int_0^{\infty} g_n(\xi)\,d\xi$.
We now conclude that $\Gamma^N\to\Gamma$, where $\Gamma$ depends self-consistently on the  boundary data ${\bar g}_n$ for $2 \leq n \leq 5$,
 and the moments $X_n$. 
\par
We are now in the same situation as in Sections \ref{S.Limit1} and \ref{S.Limit2}. In particular, analogously to the proofs of  Lemma \ref{Lem:WeakFormulation.A} and Lemma \ref{Lem:WeakFormulation.B}
we show ${\bar g}_n=g_n(0)$ for $2 \leq n \leq 5$,
$m_n=0$ for $n \geq 6$. We then infer that
$(g_n)_{n\geq2}$ provides a weak self-similar solution to the Fradkov model with $N=\infty$, which satisfies the assertions of
Corollary \ref{Cor:PositivityImpl} and Lemma \ref{Lem:NearSingularities}. Finally,
it is clear by construction that this self-similar solution satisfies the moment estimates from
Lemma \ref{Lem:savethequeen}, and thus we have finished the proof of Theorem \ref{Tmain1}. 
%
%
\section{Numerical examples}
\label{S.numerics}
%
%
%
To illustrate our analytical results we implemented the explicit Euler scheme 
for \eqref{Eqn:System.N}, which has the notable property that all
moment balances, see Lemma~\ref{Lem:Dyn.Abstract.Moment} and Corollary~\ref{Dyn:Cor:XandY}, remain valid provided that we replace the continuous time derivative 
by its discrete counterpart. In particular, computing $\Gamma$ by \eqref{Eqn:Def.MuAndGamma} our scheme
conserves the area and polyhedral defect up to computational accuracy. Moreover, due to the upwind discretization of the transport operators, and since
Lemma \ref{Dyn.Rem1} provides upper and lower bounds for $\Gamma$, one easily shows that
the explicit Euler scheme preserves the nonnegativity of the data
provided that the time step size is sufficiently small. 
\begin{figure}[ht!]
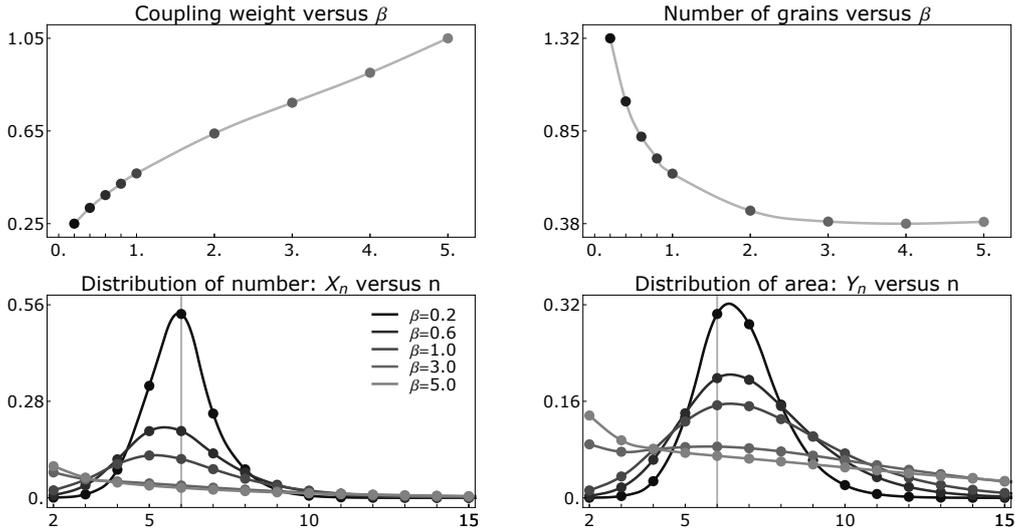
%
\centering{%
\begin{tabular}{cc}
\includegraphics[width=\figwidth, draft=\figdraft]%
{\figfile{gamma}}%
&%
\includegraphics[width=\figwidth, draft=\figdraft]%
{\figfile{tot_num}}%
\\
\includegraphics[width=\figwidth, draft=\figdraft]%
{\figfile{mom_num}}%
&%
\includegraphics[width=\figwidth, draft=\figdraft]%
{\figfile{mom_area}}%
\end{tabular}
}%
\caption{%
Self-similar solution for several values of $\beta$, where circles and lines represent numerical data and interpolating splines, respectively.  \emph{Top row}:
$\Gamma$ and $\sum_{n=2}^NX_n$ versus $\beta$. \emph{Bottom row}: $X_n$ and $Y_n$ versus $n$
with vertical line at $n=6$. }%
\label{FigNum1}%
\end{figure}%
\begin{figure}[ht!]
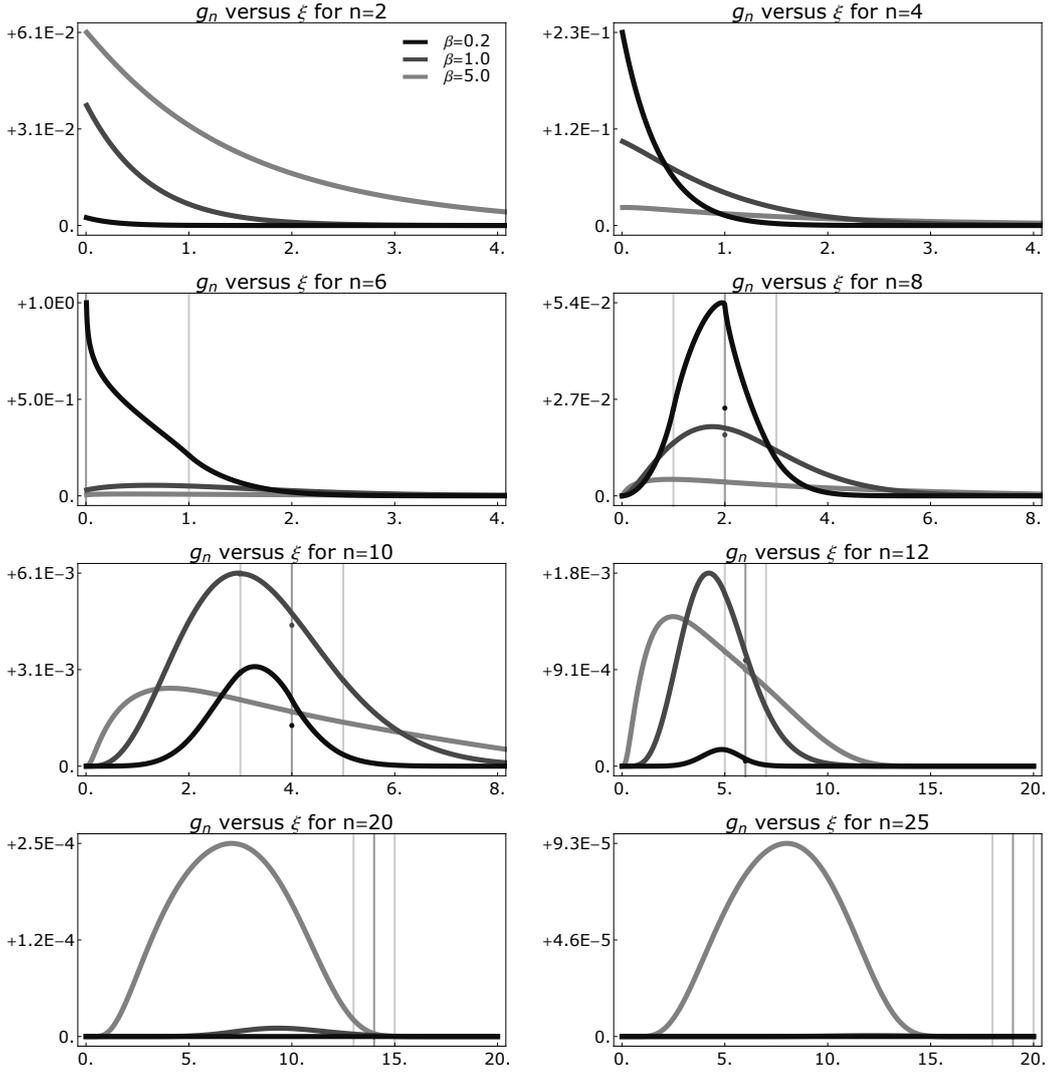
%
\centering{%
\begin{tabular}{cc}
\includegraphics[width=\figwidth, draft=\figdraft]%
{\figfile{prof_02}}%
&%
\includegraphics[width=\figwidth, draft=\figdraft]%
{\figfile{prof_04}}%
\\
\includegraphics[width=\figwidth, draft=\figdraft]%
{\figfile{prof_06}}%
&%
\includegraphics[width=\figwidth, draft=\figdraft]%
{\figfile{prof_08}}%
\\
\includegraphics[width=\figwidth, draft=\figdraft]%
{\figfile{prof_10}}%
&%
\includegraphics[width=\figwidth, draft=\figdraft]%
{\figfile{prof_12}}%
\\
\includegraphics[width=\figwidth, draft=\figdraft]%
{\figfile{prof_20}}%
&%
\includegraphics[width=\figwidth, draft=\figdraft]%
{\figfile{prof_25}}%
\end{tabular}
}%
\caption{%
Plots of $g_n$ versus $\xi$ with vertical lines at
$\xi=n-5$,  $\xi=n-6$,  and $\xi=n-7$.  
Notice that the plot range for $\xi$ varies with $n$.
}%
\label{FigNum2}%
\end{figure}%
\begin{figure}[t!]
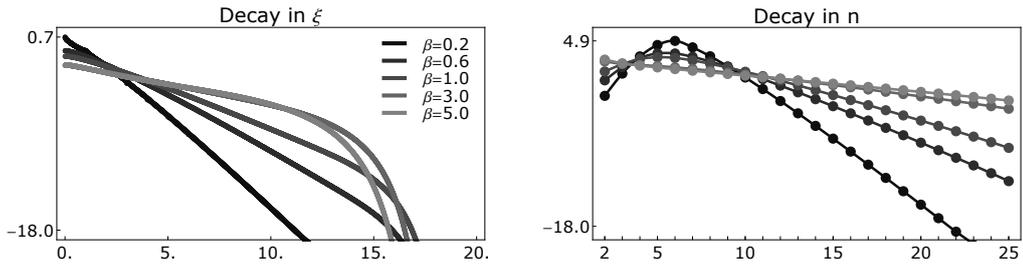
%
\centering{%
\begin{tabular}{cc}
\includegraphics[width=\figwidth, draft=\figdraft]%
{\figfile{decay_1}}%
&%
\includegraphics[width=\figwidth, draft=\figdraft]%
{\figfile{decay_2}}%
\end{tabular}
}%
\caption{%
Decay of the solutions: $\ln\bat{\sum_{n=2}^Ng_n\at{\xi}}$ versus $\xi$ and
 $\ln{X_n}$ versus $n$. %
}%
\label{FigNum3}%
\end{figure}%
\bigpar
We performed a large number of numerical simulations for various values of $\beta$ and 
different types of initial values (random, uniformly distributed, and several variants of localized data). To ensure that the initial data belong in fact to the set 
$\calU_+$, we chose at first nonnegative values $\tilde{g}_n^k$ for $n=2\tdots{N-1}$ and $k=1\tdots{K-1}$, and computed afterwards two scaling factors $\alpha_1$ and $\alpha_2$ such that
$g_n^k=\at{\chi_{2\leq{n}\leq{5}}\alpha_1+\chi_{6\leq{n}\leq{N}}\alpha_2}\tilde{g}_n^k$ satisfy the constraint $P+\eps{Q}=0$ and yield the prescribed area.
\par%
In our simulations we observed that all numerical solutions for a given value of $\beta$ converge, as $t\to\infty$, to the same steady state. We therefore conjecture, that for all $\beta$ and $N<\infty$ there exists a unique steady state that is moreover a global attractor for $\eqref{Eqn:System.N}$. It would be highly desirable to give a rigorous justification for this numerical observation, but even to prove the uniqueness of $\Gamma$ remains a challenging task.
We also conjecture that for $N=\infty$ there is only one self-similar solution with fast decay in $\xi$,
but emphasize that self-similar solutions with slow decay might exist as well.
Such solutions exist in related 
mean-field models for coarsening that couple transport and coalescence
\cite{HLN09}, but cannot be detected by our approximation scheme.
\bigpar%
The numerically computed steady states for several values of $\beta\in\ccinterval{0.2}{5.0}$ are, along with some derived data, depicted in 
Figures~\ref{FigNum1},~\ref{FigNum2}, and~\ref{FigNum3}. All computations are performed with $A=1$, $N=25$, $L=20$, and $\eps=0.05$, and due to the numerically computed residuals we expect that the discrete solutions $g_n^k$ resemble the limit profiles $g_n$ with
 $N=\infty$ very well. In particular,  Figure~\ref{FigNum3} confirms that the self-similar profiles for $N=\infty$ 
decay exponentially in $\xi$.
\par
Figure \ref{FigNum2} illustrates that for $n\geq7$ there is no pointwise convergence $g_n^k\xrightarrow{\eps\to0}g_n$ at the critical point $\xi=n-6$. In fact,  at least for small $\beta$  and moderate values of $n$ we observe that the discrete data $g_n^{k_n}$ are considerably smaller than $g_n\at{n-6}$. 
This phenomenon stems from our discretization and can be understood as follows.
On the discrete level steady states satisfy, see \eqref{Eqn:System.N},
\begin{align*}
\at{\Gamma\kappa_n-1}g_n^{k_n}&=\Gamma\at{\beta(n-1)g_{n-1}^{k_n}+
\at{\beta+1}\at{n+1}g_{n+1}^{k_n}},
\\%
\at{\Gamma\kappa_n-1}g_n^{k_n\pm{1}}-g_{n}^{k_n\pm{2}}&=
\Gamma\at{\beta(n-1)g_{n-1}^{k_n\pm{1}}+\at{\beta+1}\at{n+1}g_{n+1}^{k_n\pm{1}}}.
\end{align*}
For small $\eps$ we can express the right hand sides in terms of $g_{n\pm{1}}\at{n-6}$, the values of the limit functions $g_{n\pm1}$ at $\xi=n-6$. Equating the resulting right hand sides we then conclude that
\begin{align*}
\at{\Gamma\kappa_n-1}g_n^{k_n}=
\at{\Gamma\kappa_n-1}g_n^{k_n\pm{1}}-g_{n}^{k_n\pm{2}}+\Do{1}.
\end{align*}
If the limit function $g_n$ is also continuous at $\xi=n-6$ (according to Lemma \ref{Lem:NearSingularities}, this happens  for
 $\Gamma>2/\kappa_n=2/(2\beta+1)n$ and hence at least for large $n$) we can approximate
the terms $g_n^{k_n\pm1}$ and $g_n^{k_n\pm2}$ by $g_n\at{n-6}$. This gives
\begin{align*}
g_n^{k_n}
\quad\xrightarrow{\eps\to0}\quad%
\frac{\Gamma\kappa_n-2}{\Gamma\kappa_n-1}g_n\at{n-6}
=%
\at{1-\frac{1}{\Gamma\at{2\beta+1}n-1}}g_n\at{n-6},
\end{align*}
where the right hand side is always nonnegative due to $g_n\at{n-6}\geq0$ and $\Gamma\kappa_n>2$.
%
%
%
%
%
\paragraph*{Acknowledgements}%
We thank Francis Filbet for illuminating discussions and the Universities of Lyon 1, Oxford and Toulouse 3 for
their hospitality.
%
%

%
%
\end{document}